\documentclass[12pt]{amsart}
\usepackage{float} 
\usepackage{amsmath,amsthm,amssymb,enumitem,verbatim,stmaryrd,xcolor,microtype,graphicx,aliascnt,needspace,array}
\usepackage[T1]{fontenc}
\usepackage[utf8]{inputenc}
\usepackage[english]{babel} 
\usepackage[top=2.95cm,bottom=2.95cm,left=3.5cm,right=3.5cm]{geometry}
\usepackage[bookmarksdepth=2,linktoc=page,colorlinks,linkcolor={red!70!black},citecolor={red!80!black},urlcolor={blue!80!black}]{hyperref}

\usepackage{booktabs}
\usepackage{longtable}

\usepackage{tikz}
\usepackage{tikz-cd}\usetikzlibrary{arrows}

\usepackage[colorinlistoftodos,bordercolor=orange,backgroundcolor=orange!20,linecolor=orange,textsize=footnotesize]{todonotes} \makeatletter \providecommand \@dotsep{5} \def\listtodoname{List of Todos} \def\listoftodos{\@starttoc{tdo}\listtodoname} \makeatother 

\allowdisplaybreaks 

\usepackage{etoolbox}
\makeatletter
\patchcmd{\@startsection}{\@afterindenttrue}{\@afterindentfalse}{}{}             
\patchcmd{\part}{\bfseries}{\bfseries\LARGE}{}{}
\patchcmd{\section}{\scshape}{\bfseries}{}{}\renewcommand{\@secnumfont}{\bfseries} 
\patchcmd{\@settitle}{\uppercasenonmath\@title}{\large}{}{}
\patchcmd{\@setauthors}{\MakeUppercase}{}{}{}
\addto{\captionsenglish}{} 
\addto{\captionsenglish}{} 
\addto{\captionsenglish}{} 
\makeatother

\usepackage{fancyhdr}

\addto\extrasenglish{}  
\theoremstyle{plain}
\newtheorem{thm}{Theorem}[section] 
\newaliascnt{lemma}{thm}\newtheorem{lemma}[lemma]{Lemma}\aliascntresetthe{lemma}
\newaliascnt{cor}{thm}\newtheorem{cor}[cor]{Corollary}\aliascntresetthe{cor}
\newaliascnt{prop}{thm}\aliascntresetthe{prop}
\newtheorem{thmA}{Theorem} \newaliascnt{corA}{thmA}\newtheorem{corA}[corA]{Corollary}\aliascntresetthe{corA}
\newaliascnt{lemmaA}{thmA}\aliascntresetthe{lemmaA}
\newaliascnt{propA}{thmA}\aliascntresetthe{propA}

\newtheorem*{thm*}{Theorem}
\newtheorem*{lem*}{Lemma}
\newtheorem*{cor*}{Corollary}

\theoremstyle{definition}
\newaliascnt{df}{thm}\newtheorem{df}[df]{Definition}\aliascntresetthe{df}
\newaliascnt{rem}{thm}\newtheorem{rem}[rem]{Remark}\aliascntresetthe{rem}
\newaliascnt{ex}{thm}\aliascntresetthe{ex}

\newtheorem*{df*}{Definition}
\newtheorem*{ex*}{Example}
\newtheorem*{rem*}{Remark}

\theoremstyle{remark}

\DeclareRobustCommand{\gobblefour}[5]{}    

\DeclareFontFamily{OT1}{pzc}{}                                
\DeclareFontShape{OT1}{pzc}{m}{it}{<-> s * [1.10] pzcmi7t}{}
\DeclareMathAlphabet{\mathpzc}{OT1}{pzc}{m}{it}
\DeclareSymbolFont{sfoperators}{OT1}{bch}{m}{n} \DeclareSymbolFontAlphabet{\mathsf}{sfoperators} \makeatletter\def\operator@font{\mathgroup\symsfoperators}\makeatother 
\DeclareSymbolFont{cmletters}{OML}{cmm}{m}{it}              
\DeclareSymbolFont{cmsymbols}{OMS}{cmsy}{m}{n}
\DeclareSymbolFont{cmlargesymbols}{OMX}{cmex}{m}{n}
\DeclareMathSymbol{\myjmath}{\mathord}{cmletters}{"7C}     \let\jmath\myjmath 
\DeclareMathSymbol{\myamalg}{\mathbin}{cmsymbols}{"71}     
\DeclareMathSymbol{\mycoprod}{\mathop}{cmlargesymbols}{"60}
\DeclareMathSymbol{\myalpha}{\mathord}{cmletters}{"0B}     \let\alpha\myalpha 
\DeclareMathSymbol{\mybeta}{\mathord}{cmletters}{"0C}      \let\beta\mybeta
\DeclareMathSymbol{\mygamma}{\mathord}{cmletters}{"0D}     \let\gamma\mygamma
\DeclareMathSymbol{\mydelta}{\mathord}{cmletters}{"0E}     \let\delta\mydelta
\DeclareMathSymbol{\myepsilon}{\mathord}{cmletters}{"0F}   \let\epsilon\myepsilon
\DeclareMathSymbol{\myzeta}{\mathord}{cmletters}{"10}      \let\zeta\myzeta
\DeclareMathSymbol{\myeta}{\mathord}{cmletters}{"11}       \let\eta\myeta
\DeclareMathSymbol{\mytheta}{\mathord}{cmletters}{"12}     \let\theta\mytheta
\DeclareMathSymbol{\myiota}{\mathord}{cmletters}{"13}      \let\iota\myiota
\DeclareMathSymbol{\mykappa}{\mathord}{cmletters}{"14}     \let\kappa\mykappa
\DeclareMathSymbol{\mylambda}{\mathord}{cmletters}{"15}    \let\lambda\mylambda
\DeclareMathSymbol{\mymu}{\mathord}{cmletters}{"16}        \let\mu\mymu
\DeclareMathSymbol{\mynu}{\mathord}{cmletters}{"17}        \let\nu\mynu
\DeclareMathSymbol{\myxi}{\mathord}{cmletters}{"18}        \let\xi\myxi
\DeclareMathSymbol{\mypi}{\mathord}{cmletters}{"19}        \let\pi\mypi
\DeclareMathSymbol{\myrho}{\mathord}{cmletters}{"1A}       \let\rho\myrho
\DeclareMathSymbol{\mysigma}{\mathord}{cmletters}{"1B}     \let\sigma\mysigma
\DeclareMathSymbol{\mytau}{\mathord}{cmletters}{"1C}       \let\tau\mytau
\DeclareMathSymbol{\myupsilon}{\mathord}{cmletters}{"1D}   \let\upsilon\myupsilon
\DeclareMathSymbol{\myphi}{\mathord}{cmletters}{"1E}       \let\phi\myphi
\DeclareMathSymbol{\mychi}{\mathord}{cmletters}{"1F}       \let\chi\mychi
\DeclareMathSymbol{\mypsi}{\mathord}{cmletters}{"20}       \let\psi\mypsi
\DeclareMathSymbol{\myomega}{\mathord}{cmletters}{"21}     \let\omega\myomega
\DeclareMathSymbol{\myvarepsilon}{\mathord}{cmletters}{"22}\let\varepsilon\myvarepsilon
\DeclareMathSymbol{\myvartheta}{\mathord}{cmletters}{"23}  \let\vartheta\myvartheta
\DeclareMathSymbol{\myvarpi}{\mathord}{cmletters}{"24}     \let\varpi\myvarpi
\DeclareMathSymbol{\myvarrho}{\mathord}{cmletters}{"25}    \let\varrho\myvarrho
\DeclareMathSymbol{\myvarsigma}{\mathord}{cmletters}{"26}  \let\varsigma\myvarsigma
\DeclareMathSymbol{\myvarphi}{\mathord}{cmletters}{"27}    \let\varphi\myvarphi

\newcommand\cB{{\mathcal B}}
\newcommand\cC{{\mathcal C}}

\newcommand\cI{{\mathcal I}}

\newcommand\cL{{\mathcal L}}
\newcommand\cM{{\mathcal M}}

\newcommand\xB{{\beta}}

\renewcommand\max{\textup{max}}

\renewcommand\geq{\geqslant}
\renewcommand\leq{\leqslant}

\renewcommand{\setminus}{\backslash}

\renewcommand{\iff}{\ensuremath{\Leftrightarrow}}
\renewcommand\emptyset\varnothing

\usepackage{cancel}

\title{Matroid Bingo}

\author{Matthew Baker}
\address{\rm Matthew Baker, School of Mathematics, Georgia Institute of Technology, Atlanta, USA}
\email{mbaker@math.gatech.edu}

\author{Hope Dobbelaere}
\address{\rm Hope Dobbelaere, School of Mathematics, Georgia Institute of Technology, Atlanta, USA}
\email{hdobbelaere3@gatech.edu}

\author{Brennan Fullmer}
\address{\rm Brennan Fullmer, School of Mathematics, Georgia Institute of Technology, Atlanta, USA}
\email{bfullmer3@gatech.edu}

\author{Patrik Gajdo\v{s}}
\address{\rm Patrik Gajdo\v{s}, University of Groningen, Groningen, The Netherlands}
\email{p.gajdos@student.rug.nl}

\hypersetup{pdftitle={title...},pdfauthor={Matthew Baker, Hope Dobbelaere, Brennan Fullmer, Patrik Gajdo\v{s}}}

\begin{document}

\begin{abstract}
We investigate some natural probability distributions associated with the game of matroid bingo. 
\end{abstract}

\date{\today}
\thanks{M.B. was partially supported by NSF grant DMS-2154224 and a Simons Fellowship in Mathematics (1037306).
We thank Oliver Lorscheid and Zach Walsh for helpful discussions, and Changxin Ding, Donggyu Kim, Matt Larson, and Noah Solomon for their feedback on a preliminary copy of this manuscript. ChatGPT 5.0 provided assistance with some of the figures and tables.}

\maketitle

\section{Introduction}

\subsection{Matroid bingo}
Consider the following simplified version of the game of bingo.
Each of the $m$ players gets a card on which is printed a subset of $\{1,...,n\}$ for some positive integer $n$.
The bingo caller has a box filled with $n$ balls, on which the numbers $1$ through $n$ are written (no duplications).
The caller selects a ball at random, reads out the number, and each player who has that number on their card circles it.
The caller continues this process, selecting balls at random until someone calls out ``Bingo!'', signifying that they have circled all the numbers on their card.

The company that manufactures the cards (let us call it Matroid, Inc.) follows three simple rules in the design process:
\begin{enumerate}
    \item[(B1)] No bingo card is empty. In other words, no one shouts ``Bingo!'' before the game even starts.
    \item[(B2)] Every player has a nonzero chance of winning. In other words, no card is a proper subset of another one.
    \item[(B3)] There is no possibility of a tie. In other words, it can never happen that two players shout ``Bingo!'' at the same time; this guarantees that the game always has a unique winner.
\end{enumerate}

For example, consider the following set of cards (with $m=5$ players and $n=8$ possible numbers):

\begin{figure}[h]
    \centering
    \includegraphics[scale = 0.5]{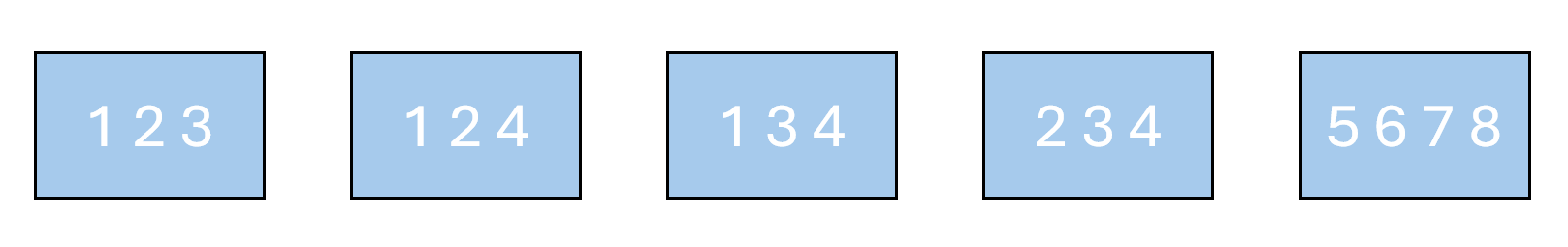}
    \caption{A valid set of matroid bingo cards.}
    \label{fig:cards-1}
\end{figure}

It is not difficult to check that this particular set of cards satisfies properties (B1) through (B3).

\medskip

On the other hand, the following is {\bf not} a valid set of bingo cards, because if the bingo caller happens to draw the numbered balls in numerical order $1,2,3$, both players will shout ``Bingo!'' at the same time.

\begin{figure}[h]
    \centering
    \includegraphics[scale = 0.5]{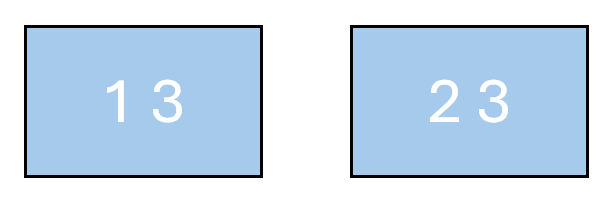}
    \caption{A set of cards for which a tie is possible, and therefore doesn't comprise a valid bingo game.}
    \label{fig:enter-label}
\end{figure}

Despite the simplicity of the rules, it turns out that this game encodes --- in a precise way --- the axioms for a \emph{matroid}.\footnote{More precisely, a set ${\mathcal C}$ of bingo cards is valid if and only if ${\mathcal C}$ is the collection of circuits of a matroid on $\{ 1,\ldots, n \}$.} 
For this reason, we will refer to the game as \textbf{matroid bingo}.

\subsection{Matroids}

For the purposes of this paper, we will use the following characterization\footnote{Matroids can famously be characterized using numerous different systems of axioms. Axiomatic descriptions that characterize the same objects but in a non-obvious way are known in the field as \emph{cryptomorphisms}.} of matroids:


\begin{df} A \textbf{matroid} $\cM$ is a finite set $E$ together with a collection $\cC$ of subsets of $E$, called the \textbf{circuits} of the matroid, such that:
\begin{enumerate}
    \item[(C1)] $\emptyset \notin \cC$.
    \item[(C2)] If $C_1, C_2 \in \cC$ and $C_1 \subseteq C_2$, then $C_1 = C_2$.
    \item[(C3)] If $C_1, C_2 \in \cC$, $C_1 \neq C_2$, and $e\in C_1 \cap C_2$, there exists $C_3 \in \cC$ such that $C_3 \subseteq (C_1 \cup C_2) \backslash \{e\}$.
\end{enumerate}
\end{df}

We can readily verify that given the set $E = \{ 1,\ldots, n \}$, a collection $\cC$ of subsets of $E$ (identified with bingo cards) satisfies axioms (B1)-(B3) iff $\cC$ satisfies (C1)-(C3).
Indeed, it is clear that (B1) corresponds to (C1) and (B2) corresponds to (C2), so the only issue is to relate (B3) and (C3).
If $\cC$ satisfies (C3), then there cannot be a tie in matroid bingo, since if calling out a shared element $e$ belonging to distinct cards $C_1$ and $C_2$ were to complete both $C_1$ and $C_2$, (C3) guarantees that there is another bingo card $C_3 \subseteq (C_1 \cup C_2) \backslash \{e\}$ which was already complete before $e$ was called out. Conversely, if $\cC$ satisfies (B3), then given $C_1 \neq C_2$ in $\cC$ and $e \in C_1 \cap C_2$, if the bingo caller were to call out the numbers belonging to $C_1 \cup C_2 \backslash \{ e \}$ before all the other numbers, followed immediately by $e$, there would be a tie---violating (B3)---unless some other card $C_3$ had already been completed. This means precisely that $C_3$ is a circuit of $\cM$ with $C_3 \subseteq (C_1 \cup C_2) \backslash \{e\}$.

\medskip

The game of matroid bingo was invented by Dima Fon-Der-Flaass (cf.~\cite[p.5]{CameronNotes} and \cite[p.33]{Borovik-Gelfand-White}). Although it is arguably the simplest and most intuitive way to explain the concept of a matroid to a non-mathematician, the game does not appear to be widely known, even within the matroid community. 

\subsection{Monotonicity}
\label{sec:intro-monotonicity}

Our goal in this paper is to explore some questions which arise quite naturally when one starts to think about matroids from the point of view of matroid bingo. For example, take a look at \autoref{fig:cards-1} again. If you were given a choice of any of these cards to play, which one would you pick?

\medskip

It is tempting to choose one of the cards with three numbers, rather than four, simply because the list of numbers to be circled is shorter. However, upon further inspection it is not clear that this is the right decision, since all the short cards have a lot of overlap with one another, creating more competition when the numbers on these cards are called out. It turns out that the optimal choice---the one which leads to the highest probability of winning---is to pick the card $\{ 5,6,7,8 \}$ with four numbers on it (cf. \autoref{tab:matroid-1}). Whenever a longer card beats or ties with a shorter one, we call it a \textbf{monotonicity violation}.

\medskip

Although monotonicity violations exist, they are relatively rare. 
For cards using only single-digit numbers, we systematically tested all possibilities and found 
precisely 11 instances of monotonicity violations out of the 385,360 isomorphism classes of matroids with at most 9 elements (see \autoref{Appendix A} for the complete list).\footnote{We have put together a GitHub repository containing more detailed information about our computer calculations, see \url{https://github.com/MathMayhem/Matroid-Bingo}. The matroids were gathered from an online database created by Dillon Mayhew and Gorden F. Royle \cite{mayhew2008matroids}.}


\medskip

The reader might wonder if the monotonicity violation depicted in \autoref{fig:cards-1} is an artifact of the card $\{ 5,6,7,8 \}$ being disjoint from all the other cards, which only involve the numbers 1 through 4. (In the language of matroids, this corresponds to the fact that the corresponding matroid is disconnected.) However, the following example shows that the situation is more complicated than this:

\begin{figure}[h]
    \centering
    \includegraphics{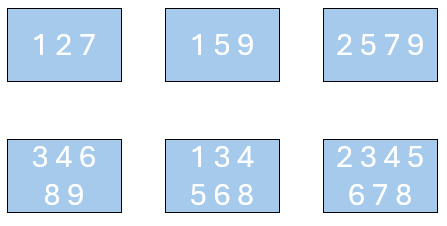}
    \caption{A set of bingo game cards corresponding to a connected matroid with a monotonicity violation.}
    \label{fig:cards-3}
\end{figure}

It turns out that, if given a choice between $\{ 2, 5, 7, 9 \}$ and $\{ 3, 4, 6, 8, 9 \}$, you are better off choosing $\{ 3, 4, 6, 8, 9 \}$, as can be seen in \autoref{tab:intro table}.

\begin{figure}
\begin{center}
\begin{tabular}{| c | r @{\,$\approx$\,} l |}
    \hline
    \multicolumn{3}{|c|}{\textbf{Winning probabilities $\xB_C$}} \\
    \multicolumn{3}{|c|}{for the cards in \autoref{fig:cards-3}} \\
    \hline
    $C$ & \multicolumn{2}{c|}{$\xB_C$} \\
    \hline
    127      & $19/60$  & 0.3167 \\
    159      & $32/105$ & 0.3048 \\
    34689    & $37/252$ & 0.1468 \\
    2579     & $13/90$  & 0.1444 \\
    134568   & $5/84$   & 0.0595 \\
    2345678  & $1/36$   & 0.0278 \\
    \hline
\end{tabular}
\end{center}
    \caption{}
    \label{tab:intro table}
\end{figure}

\medskip

\subsection{Winning probabilities}
\label{sec:intro-winningprobabilities}
Let us formulate things more precisely now. Given a legal set of bingo cards---or, equivalently, the set $\cC$ of circuits of a matroid $\cM$---we denote by $\xB_C$ the probability that a given circuit $C$ wins the game. By assumption, $\xB_C > 0$ for all $C \in \cC$, and 
\[
\sum_{C \in \cC} \xB_C = 1.
\]



The following two theorems allow us to calculate these probabilities. We have found the second formula to be more useful, both for computer calculations and theoretical results, but we include both formulas for completeness.
In the statements, and elsewhere in the paper, we write $[n]$ as shorthand for the set $\{ 1,2,\ldots,n \}$.
Also, for a matroid $\cM$, the maximum size of a subset of $\{ 1,2,\ldots,n \}$ which does not contain any circuit is called the {\bf rank} of $\cM$, and is denoted $r(\cM)$.

\begin{thmA}
Given a matroid $\cM$ on $E = [n]$ and a circuit $C$ of $\cM$,
$$\xB_C = \sum_{k = 0}^{|\cC| - 1} \left( (-1)^k \cdot \sum_{\substack{S \subseteq \cC \backslash \{ C \}, \\ |S| = k}} \frac{|C|}{|(\cup_{C' \in S} C') \cup C|} \right).$$
\end{thmA}

\begin{thmA} \label{thm:mainbetaformula}
Given a matroid $\cM$ on $E = [n]$ and a circuit $C$ of $\cM$, we have
\[\xB_C=\sum_{t=|C|}^{r(\cM)+1} \xB_C(t),\]
where
\begin{equation} \label{eq:introbetaformula1}
\xB_C(t) = \frac{|C|}{n} \cdot |\cI_{C, \, t - |C|}| \cdot \binom{n - 1}{t - 1}^{-1}.
\end{equation}

Here $\cI_{C, k}$ denotes the collection of all $k$-element subsets $S$ of $E \backslash C$ such that $C$ is the unique circuit contained in $S \cup C$.
\end{thmA}

\begin{rem} \label{rem:Tuttepoly}
The right-hand side of \eqref{eq:introbetaformula1} has a simple reformulation in terms of the \emph{Tutte polynomial} of the contracted matroid $\cM/C$, 
cf.~\autoref{thm:Tuttepolyformula}.
\end{rem}

\begin{rem}
The quantity $\xB_C(t)$ appearing in \eqref{eq:introbetaformula1} is the probability that a given bingo card $C$ wins in the $t^{\rm th}$ round of the game, where $t=1,2,\ldots,n$. We call these the \textbf{timed winning probabilities}.
A matroid bingo game always ends by round $r+1$, where $r$ is the rank of the corresponding matroid, so $\xB_C(t)=0$ for $t>r+1$. Similarly, the card \(C\) cannot win before at least \(|C|\) numbers have been called out, so \(\xB_C(t)=0\) for \(t<|C|\).
See \autoref{fig:Matroid 9 Timed Beta Values} for an example.
\end{rem}

\begin{figure}[h]
    \centering
    \includegraphics[width=0.8\linewidth]{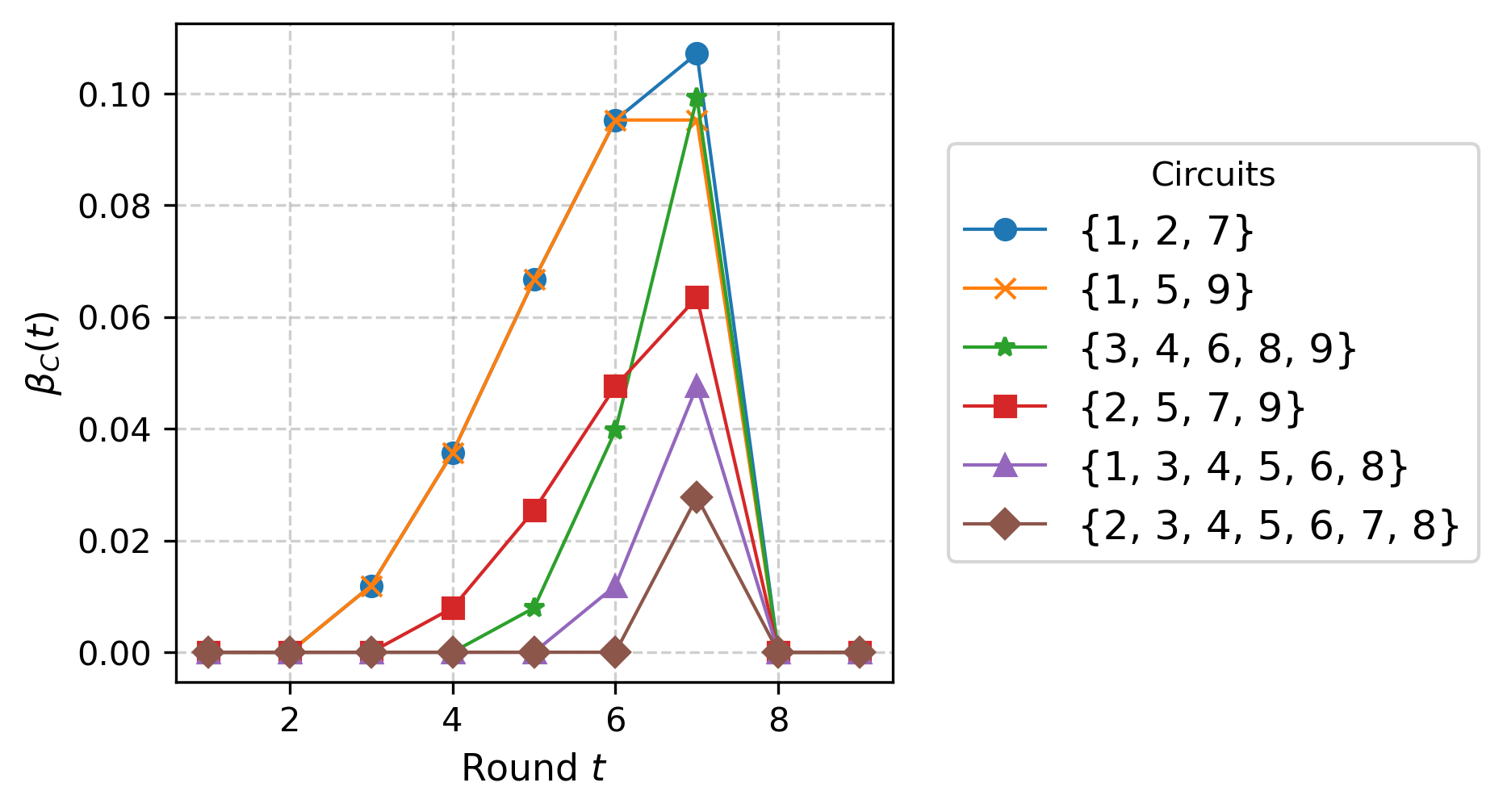}
    \caption{Timed winning probabilities for the matroid in \autoref{fig:cards-3}.}
    \label{fig:Matroid 9 Timed Beta Values}
\end{figure}


A sequence $a_1,\ a_2,\ldots,\ a_n$ of real numbers is \textbf{log-concave} if for any three consecutive terms, the middle term squared is at least as big as the product of its neighbors:
\[
a_i^2 \ge a_{i-1} \cdot a_{i+1}.
\]

Combining \autoref{thm:mainbetaformula} with a deep recent result from combinatorial Hodge theory, we prove the following (\autoref{thm:logconcave}).

\begin{thmA} \label{thm:intro-logconcave}
For every matroid $\cM$ and every circuit $C$ of $\cM$, the sequence 
\[
\xB_C(1),\ \xB_C(2),\ldots,\ \xB_C(n)
\]
of timed winning probabilities is log-concave. 
\end{thmA}

It is well-known (see, e.g., \cite{unimodal}) that a log-concave sequence $a_0,a_1,a_2,\ldots,a_n$ of non-negative numbers with no internal zeros\footnote{A sequence $a_0,a_1,a_2,\ldots,a_n$ has \textbf{no internal zeros} if whenever $a_i = 0$ for some index $i$, either all terms before it or all terms after it are also zero, i.e., the set of indices $i$ where $a_i \neq 0$ forms a contiguous block of integers.}
is \textbf{unimodal}, meaning that the terms weakly increase up to a certain point and then weakly decrease, i.e.,
there exists an index $m$ such that $a_0 \le a_1 \le \cdots \le a_m$ and $a_m \ge a_{m+1} \ge \cdots \ge a_n$.
One deduces from this that the sequence 
$\xB_C(1),\ \xB_C(2),\ldots,\ \xB_C(n)$
is unimodal. 
Although intuitively plausible, this is a highly non-trivial fact about matroid bingo.\footnote{From a purely pedagogical perspective, this might also be the simplest way to explain a particular instance of combinatorial Hodge theory to a non-mathematician!}

\subsection{Upper and lower bounds on winning probabilities}






We will prove the following upper and lower bounds for $\xB_C$ (\autoref{thmBC} and \autoref{cor:thmD-simplified}):

\begin{thmA}\label{Upper bound}
Let $\cM$ be a matroid of rank $r$ on $n$ elements and let $C$ be a circuit of \(\cM\). 
Then,
$$\xB_C \leq \frac{\binom{r+1}{|C|}}{\binom{n}{|C|}} = \frac{(r + 1)! \cdot (n - |C|)!}{n! \cdot (r+1-|C|)!}.$$
\end{thmA}

\begin{thmA}\label{intro-thmD}
Let $\cM$ be a matroid of rank $r$ on $n$ elements and let $C$ be a circuit of \(\cM\). Let $d = r+1 - |C|$. Then, 
$$\xB_C \geq \binom{n - d}{|C|}^{-1} = \frac{|C|!(n-d+|C|)!}{(n-d)!}.$$
\end{thmA}

\begin{rem}
The upper bound in \autoref{Upper bound} is sharp. Specifically, for any $n$, $r$, and $d$ with $0 \leq d \leq r < n$, there is a matroid $\cM$ of rank $r$ on $n$ elements, together with a circuit $C$ of size $r + 1 - d$ in $\cM$, which achieves the upper bound
(cf.~\autoref{rem:upper strict}).
The lower bound in \autoref{intro-thmD} is not sharp for \(d>0\); however, it is deduced from a stronger bound (\autoref{thmD}) which is sharp in the same sense as above (cf.~\autoref{rem:lower strict}).
\end{rem}

When $d=0$, the upper bound in \autoref{Upper bound} and the lower bound in \autoref{intro-thmD} coincide, so we obtain:
\begin{corA}
If $|C|=r+1$, then $\xB_C=\binom{n}{r+1}^{-1}.$
\end{corA}

In other words, bingo cards with the maximum possible size $r+1$ all have the same winning probability.


\subsection{Monotonicity violations revisited}

Returning to our earlier theme of monotonicity violations, we have the following positive results.
In the statements, we say that \textbf{monotonicity holds} for a matroid $\cM$
if there is no monotonicity violation among the circuits of $\cM$, i.e., $\xB_{C_1} <  \xB_{C_2}$ whenever $C_1,C_2 \in \cC$ satisfy $|C_1| > |C_2|$.

\medskip

From \autoref{Upper bound} and \autoref{intro-thmD}, we will deduce the following result (\autoref{corA}):

\begin{corA} \label{cor:intro-noviolationsforrplus1}
If $\cM$ is a matroid of rank $r$ and $C_1,C_2$ are circuits of $\cM$ with
$|C_1| = r+1$ and $|C_2| < r+1$, then $\xB_{C_1} <  \xB_{C_2}$. 
In other words, there can be no monotonicity violations involving a circuit of maximum size, $r+1$.
\end{corA}

\begin{rem}
In particular, if $\cM$ is a rank $r$ \textbf{paving matroid} (i.e., a matroid whose circuits all have size $r$ or $r+1$), there are no monotonicity violations. A well-known conjecture in matroid theory \cite[Conjecture 1.6]{Mayhew-et-al-2011} asserts that, from an asymptotic point of view, 100\% of all matroids are paving. We therefore conjecture that 100\% of all matroids have no monotonicity violation.
\end{rem}

We also show that monotonicity always holds when $n \gg r$ (\autoref{thm:threshhold}):

\begin{thmA} \label{thm:intro-threshhold}
For fixed $r$, there exists $N = N(r)$ so that if $\cM$ is a matroid of rank $r$ on $n \geq N$ elements, then 
monotonicity holds for $\cM$.
\end{thmA}

\begin{rem}
More precisely, one can make the argument below quantitative and show that it is possible to take 
$N(r) = (r + 2) + (r + 1)! \, \big\lceil \, e^{3r + \frac{25}{6}} \, \big\rceil$
in the statement of the theorem.
\end{rem}


\subsection{Equitable matroids}

Although axiom (B2) asserts that every player in matroid bingo has a non-zero chance of winning, the game is rarely ``fair''.\footnote{We note, wistfully, that the same can be said for life itself.}
In a truly fair game, every player would have an equal chance of winning. We call the matroids for which this holds \textbf{equitable}.\footnote{The reader should not confuse this notion of ``equitable'' with other uses of the term in the matroid theory literature, for example in Mayhew's paper \cite{Mayhew2006}. One can profitably think of Mayhew's notion as \emph{basis equitable} and ours as \emph{circuit equitable}. Our notion of equitability appears to be unrelated to the one considered in \cite{akrami2025matroidsequitable}.} 

We summarize our findings about equitable matroids in the following:

\begin{thmA}\label{thm:intro-equitable}
\begin{enumerate}
    \item[]
    \item If $C_1,C_2$ are circuits of a matroid $\cM$ and there is an automorphism $\varphi$ of $\cM$ with $\varphi(C_1)=C_2$, then $\xB_{C_1} = \xB_{C_2}$. In particular, if the automorphism group of $\cM$ acts transitively on the set of circuits then $\cM$ is equitable.
    \item There exists an equitable matroid $\cM$ whose automorphism group does not act transitively on the set of circuits.
    \item Uniform matroids and duals of projective geometries over finite fields are equitable.
\end{enumerate}
\end{thmA}


It is an interesting and seemingly difficult question to classify all equitable matroids.

\section{Proofs}


\subsection{Background from matroid theory}
\label{sec:matroids}

In this section, we review some definitions and facts from matroid theory which will be used later on.

\begin{df*}
Let $\cM$ be a matroid.
\begin{enumerate}
    \item An \textbf{independent set} of $\cM$ is a subset $I$ of $E$ which does not contain any circuit. A subset of $E$ which is not independent is called \textbf{dependent}. The collection of independent sets in a matroid will be denoted $\mathcal{I}$.
    \item A \textbf{basis} of $\cM$ is a maximal independent set $I$, i.e., an independent set $I$ such that $I \cup \{ e \}$ contains a circuit for all $e \in E \backslash I$.
    \item The \textbf{rank function} $r: 2^E \to \mathbb{Z}_{\geq 0}$ of $\cM$ is defined by 
    \[
    r(X) = \max\{|I| : I \subseteq X, I \in \mathcal{I}\}
    \]
    for all subsets $X \subseteq E$. 
    \item The rank of $E$ is called the \textbf{rank of the matroid}, and denoted $r(M)$ (or just $r$ when the context is clear).
\end{enumerate}
\end{df*}

\begin{lemma} 
\cite[Lemma 1.2.1]{oxley2011}
If $\cM$ is a matroid, any two bases have the same cardinality, namely $r(M)$. 
\end{lemma}

The following result is known as the \textbf{basis exchange property}:

\begin{lemma} \label{lem:basisexchange}
\cite[Lemma 1.2.2]{oxley2011} 
If $B_1$ and $B_2$ are distinct bases of a matroid $\cM$ and $x \in B_1 \backslash B_2$, then there exists an element $y \in B_2 \backslash B_1$ such that $(B_1 \backslash \{x\}) \cup \{y\}$ is also a basis of $\cM$. 
\end{lemma}

\begin{df*}
The \textbf{uniform matroid} of rank $r$ on $n$ elements, denoted $U_{r, n}$, is the matroid with ground set $E = \{1, ..., n\}$ and circuit set $\cC = \{C \subseteq E : |C| = r + 1\}$.
\end{df*}

\begin{df*}
The \textbf{direct sum} of matroids $\cM_1$ and $\cM_2$ on the ground sets $E_1$ and $E_2$ is the matroid $\cM = \cM_1 \oplus \cM_2$ on $E = E_1 \sqcup E_2$ with circuit set $\cC = \cC_1 \cup \cC_2$. 
\end{df*}

\begin{df*}
Given a multigraph $G = (V, E)$, we can create a matroid $\cM = \cM[G]$ on $E$ by letting the circuits of $\cM$ be the simple cycles of $G$.
Matroids of this form are called \textbf{graphic matroids}.
\end{df*}

The matroid depicted in \autoref{fig:cards-1} is not graphic. Indeed, if $\{123\}$ and $\{124\}$ are circuits of a graphic matroid $\cM[G]$, then $\{34\}$ must also be a circuit, but $\{34\}$ is not one of the bingo cards. On the other hand, the matroid depicted in \autoref{fig:cards-3} is graphic; a graphic representation is given in \autoref{fig:graph-1}.

\begin{figure}[h]
    \centering
    \includegraphics[scale = 0.5]{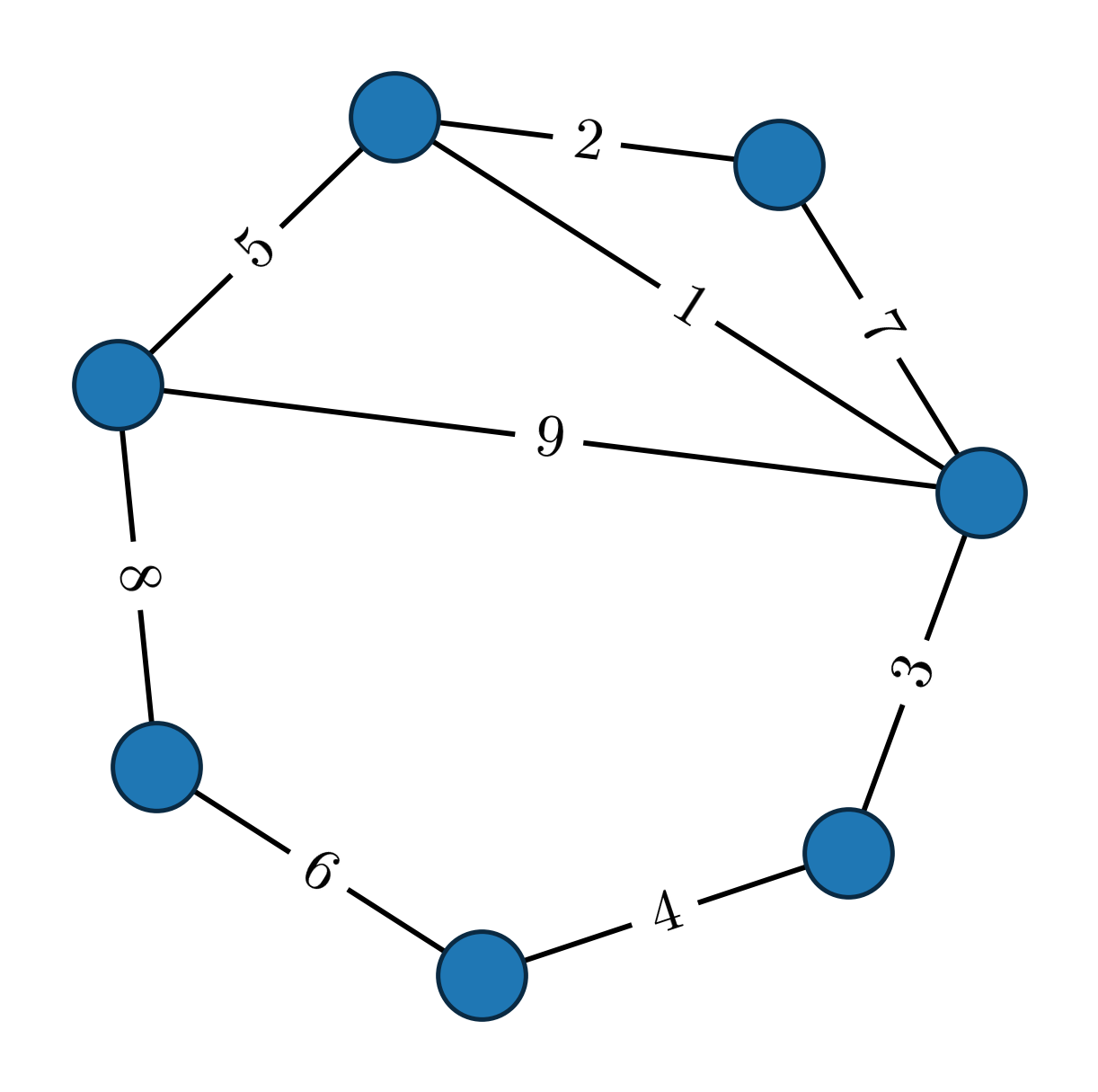}
    \caption{A graphic representation of the matroid depicted in \autoref{fig:cards-3}.}
    \label{fig:graph-1}
\end{figure}

The following result summarizes one particular consequence of the well-known fact (see \cite[Chapter 1]{oxley2011}) that matroids have ``cryptomorphic'' descriptions in terms of circuits, independent sets, bases, or rank functions.

\begin{thm} \label{thm:cryptocor}
A matroid is determined by any of: (a) its collection of independent sets; (b) its collection of bases; or (c) its rank function.
\end{thm}

One use of \autoref{thm:cryptocor} is that it enables us to easily describe the \emph{dual} of a matroid:
\begin{df*}
Given a matroid $\cM$ on $E$ with set of bases $\cB$, its \textbf{dual matroid} is the matroid $\cM^*$ on $E$ whose set of bases is 
\[
\cB^* := \{ E \backslash B \; : \; B \in \cB \}.
\]
\end{df*}

By construction, we have $\cM^{**} = \cM$.

\medskip

We will also need the concept of loops and coloops.

\begin{df*}
Let $\cM$ be a matroid.
\begin{enumerate}
\item A \textbf{loop} of $\cM$ is a single-element circuit. Equivalently, a loop is an element of $E$ which is not contained in any basis.
\item A \textbf{coloop} of $\cM$ is an element $e \in E$ not contained in any circuit; equivalently, a coloop is an element of $E$ which is contained in every basis.
\end{enumerate}
\end{df*}

\begin{rem}
It is easy to see that $e$ is a coloop of $\cM$ iff it is a loop of $\cM^*$, and vice-versa.
\end{rem}

\begin{lemma} \label{lem:r+1-circuit-coloopless}
If a matroid $\cM$ of rank $r$ has a circuit $C$ of size $r+1$, then $\cM$ is coloopless.
\end{lemma}

\begin{proof}
Since $C$ is a circuit, it is a minimally dependent set. Hence for each $e\in C$,
the set $C\setminus\{e\}$ is independent. But $|C\setminus\{e\}|=r$, so
$B := C\setminus\{e\}$ is a basis of $\mathcal{M}$.

Now let $x\in E$ be arbitrary.
\begin{itemize}
\item If $x\in C$, then the basis $C\setminus\{x\}$ does not contain $x$.
\item If $x\notin C$, pick any $e\in C$; the basis $C\setminus\{e\}$ does not contain $x$ either.
\end{itemize}
Thus, for every $x\in E$, there exists a basis that omits $x$. Therefore no element
lies in every basis, i.e., $\mathcal{M}$ has no coloops.
\end{proof}

We will also make use of the fundamental operations of deletion and contraction.

\begin{df*}
Let $\cM$ be a matroid on $E$ and let $T \subseteq E$.
\begin{enumerate}
\item The \textbf{deletion} $\cM \backslash T$ is the matroid on $E \backslash T$ whose circuits are $\{C \subseteq E \backslash T : C \in \mathcal{C} \}$.
\item The \textbf{contraction} $\cM / T$ is the matroid on $E \backslash T$ whose circuits are the non-empty inclusion-minimal elements of $\{C \backslash T : C \in \mathcal{C}\}$.
\end{enumerate}
\end{df*}

\begin{rem}
By \cite[Statement 3.1.1]{oxley2011}, deletion and contraction are dual operations: $(\cM \backslash T)^* = \cM^* / T$.
\end{rem}

Successive contractions and deletions commute with one another:

\begin{lemma} \label{lem:minors} \cite[Proposition 3.1.25]{oxley2011}
Let $\cM$ be a matroid on $E$, and let $T_1$ and $T_2$ be disjoint subsets of $E$. Then:
\begin{enumerate}
    \item $(\cM\backslash T_1)\backslash T_2 = \cM\backslash (T_1 \cup T_2) = (\cM \backslash T_2) \backslash T_1$.
    \item $(\cM / T_1)/T_2 = \cM/(T_1 \cup T_2) = (\cM/T_2)/T_1$.
    \item $(\cM \backslash T_1)/T_2 = (\cM/T_2)\backslash T_1$.
\end{enumerate}
\end{lemma}

\begin{lemma}
\label{lem:basisrankcontract}
\cite[Section 1.3 \& Proposition 3.1.7]{oxley2011} 
\begin{enumerate}
\item The independent sets (resp. bases) of $\cM \backslash T$ are the independent sets (resp. bases) of $\cM$ not containing elements of $T$. If $T$ is independent in $\cM^*$, then $r(\cM \backslash T) = r(\cM)$.

\item If $T$ is independent in $\cM$, the independent sets (resp. bases) of $\cM / T$ are those subsets $X \subseteq E \backslash T$ such that $X \cup T$ is independent in $\cM$ (resp. is a basis of $\cM$). Moreover, $r(\cM / T) = r(\cM) - |T|$.
\item On the other hand, if $e$ is a coloop of $\cM$ then $r(\cM \backslash \{ e \})=r(\cM) - 1$, and if $e$ is a loop of $\cM$ then $r(\cM / \{e\}) = r(\cM)$.

\end{enumerate}
\end{lemma}

\begin{cor} \label{cor:contracting-a-circuit}
If $C$ is a circuit of $\cM$ then $r(\cM / C) = r(\cM)-|C|+1$.
\end{cor}

\begin{proof}
Let $I = C \backslash \{ e \}$ for some $e \in C$. Since $C$ is a minimal dependent set, $I$ is independent in $\cM$. 
We have $\cM / C = (\cM / I) / \{ e \}$. By \autoref{lem:basisrankcontract}(2), $r(\cM / I)=r(\cM)-(|C|-1)$.
Moreover, one verifies easily that $e$ is a loop of $\cM / I$, and thus $r(\cM / C)=r(\cM / I)$ by \autoref{lem:basisrankcontract}(3).
\end{proof}

\begin{lemma} \label{lem:coloopless}
If $\cM$ is a coloopless matroid on the ground set $E$ and $e\in E$, then the contraction $\cM/\{e\}$ is also coloopless.
\end{lemma}

\begin{proof}
Suppose, for the sake of contradiction, that $f\in E\setminus\{e\}$ is a coloop of $\cM/\{e\}$.
By \autoref{lem:basisrankcontract}, the bases of $\cM/\{e\}$ are exactly the sets $B\setminus\{e\}$, where
$B$ ranges over the bases of $\cM$ with $e\in B$. Thus $f$ being a coloop of $\cM/\{e\}$
means:
\begin{equation}\label{eq:all-contain-f}
\text{for every basis $B$ of $M$ with $e\in B$, we have $f\in B$.}
\end{equation}

Since $\cM$ is coloopless, $f$ is not a coloop of $\cM$, so there exists a basis
$B_0$ of $\cM$ with $f\notin B_0$. If $e\in B_0$, then $B_0\setminus\{e\}$ is a
basis of $\cM/\{e\}$ that does not contain $f$, contradicting our assumption that $f$ is a coloop
of $\cM/\{e\}$. Hence we may assume $e\notin B_0$.

Consider $B_0\cup\{e\}$; this set is dependent, so it contains a circuit
$C\subseteq B_0\cup\{e\}$ with $e\in C$. Since $f\notin B_0$ and $f\ne e$, we
have $f\notin C$. By \autoref{lem:basisexchange} (the basis exchange property), choosing any $x\in C\cap B_0$
yields a basis
\[
B_1 \;=\; (B_0-\{x\})\cup\{e\}
\]
of $M$. This basis $B_1$ contains $e$ and, because $f\notin B_0$ and $f\ne e$,
also satisfies $f\notin B_1$. This contradicts \eqref{eq:all-contain-f}.
Therefore no such $f$ exists, and $M/\{e\}$ is coloopless.
\end{proof}

Intuitively, numbers that do not appear on any bingo cards have no effect on the winning probabilities of cards in matroid bingo. So we may as well ignore them. The following result, whose proof is a routine translation of this intuition (and is therefore omitted), makes this precise:

\begin{lemma} \label{coloops-don't-change-beta-values}
If $e$ is a coloop in $\cM$ and $C$ is a circuit of $\cM$ (necessarily not containing $e$), let $C'$ be the corresponding circuit of $\cM / e $.
Then $\beta_C=\beta_{C'}$.
\end{lemma}

\begin{rem}
Note that Lemma~\ref{coloops-don't-change-beta-values} only applies to \emph{untimed} winning probabilities. Fixing a time $t$, it is frequently the case that $\beta_C(t) \neq \beta_{C'}(t)$.  
\end{rem}

We will also make use of the \emph{Tutte polynomial}, a celebrated invariant of matroids that satisfies a useful deletion-contraction recurrence.

\begin{df*}
The \textbf{Tutte polynomial} \(T_{\cM}(x,y)\) of a matroid \(\cM\) is defined as:
\begin{equation} \label{eq:tutte-def}
T_\cM(x, y) = \sum_{A \subseteq E} (x-1)^{r(E) - r(A)} (y-1)^{|A| - r(A)}.
\end{equation}
\end{df*}

\begin{lemma} \label{lem:coloop tutte}
\cite[p.267]{welsh2010matroid}
Let $\cM$ be a matroid on $E$ and let $e \in E$.
\begin{enumerate}
    \item If $e$ is neither a loop nor coloop of $\cM$, then
$$T_{\cM}(x, y) = T_{\cM \backslash e}(x, y) + T_{\cM/e}(x, y).$$
    \item If $e$ is a coloop of $\cM$, then
$$T_{\cM}(x, y) = xT_{\cM \backslash e}(x, y).$$
    \item If $e$ is a loop of $\cM$, then
$$T_{\cM}(x, y) = yT_{\cM / e}(x, y).$$
\end{enumerate}
\end{lemma}

\begin{rem} 
By \cite[p.267]{welsh2010matroid}, we have $T_{\cM^*}(x,y) = T_{\cM}(y,x)$.
\end{rem}

Finally, we will need the following basic definitions.

\begin{df*} \cite[Proposition 4.1.3]{oxley2011}
A matroid $M$ on $E$ is \textbf{connected} if, for every pair of distinct elements of $E$, there is a circuit containing both.
\end{df*}

\begin{df*}
A \textbf{matroid isomorphism} from $\cM_1$ to $\cM_2$ is a bijection $f: E_1 \rightarrow E_2$ between the ground sets of $\cM_1$ and $\cM_2$, respectively, such that 
the induced map $g: \cC_1 \rightarrow \cC_2$ between circuit sets is also a bijection. If there exists a matroid isomorphism from $\cM_1$ to $\cM_2$, we say that $\cM_1$ and $\cM_2$ are \textbf{isomorphic}.
If $\cM_1 = \cM_2 = \cM$, then $f$ is called an \textbf{automorphism} of $\cM$.
\end{df*}

\begin{df*}
A \textbf{minor} of $\cM$ is any matroid isomorphic to a matroid $\cM'$ obtained from $\cM$ by a sequence of deletions and contractions.   
\end{df*}


\noindent
\textbf{Notation.} For a polynomial $f$, let $[x^k] f(x)$ denote the coefficient of $x^k$ in $f(x)$.

\subsection{Formulas for winning probabilities}

\medskip




\begin{thm}
Given a matroid $\cM$ on $E = [n]$ and a circuit $C$ of $\cM$,
$$\xB_C = \sum_{k = 0}^{|\cC| - 1} \left( (-1)^k \cdot \sum_{\substack{S \subseteq \cC \backslash \{C\}, \\ |S| = k}} \frac{|C|}{|(\cup_{C' \in S} C') \cup C|} \right).$$
\end{thm}

\begin{proof}
For a circuit $C' \in \cC$ and a permutation $\sigma$ of $E$, define $t_{C'}(\sigma)$ to be the smallest $t \in \{1,\ldots,n \}$ such that $C' \subseteq \sigma(\{ 1,\ldots, t \})$, and
define $P_{C'}$ to be the set of all permutations $\sigma$ such that $t_{C'}(\sigma) \leq t_{C}(\sigma)$. Note that it is possible to have a permutation $\sigma$ and circuits $C_1 \neq C_2$ with $t_{C_1}(\sigma) = t_{C_2}(\sigma)$ (this doesn't contradict axiom (B3) of the matroid bingo game, because that axiom only applies to the first completed card in the game). Define $P$ to be the set of all permutations of $E$.

Note that $P \backslash \cup_{C' \in \cC \backslash \{C\}} P_{C'}$ consists of those permutations $\sigma$ for which $t_{C'}(\sigma) > t_C(\sigma)$ for all circuits $C' \neq C$. So, by definition, we have 
\begin{equation} \label{eq:betaformula1}
\beta_C = \frac{1}{n!} \cdot |P \backslash \cup_{C' \in \cC \backslash \{C\}} P_{C'}| = 1 - \frac{|\cup_{C' \in \cC \backslash \{C\}} P_{C'}|}{n!}.
\end{equation}

By the inclusion-exclusion principle, we have
\begin{equation} \label{eq:inclusion-exclusion}
|\cup_{C' \in \cC \backslash \{C\}} P_{C'}| = \sum_{k = 1}^{|\cC| - 1} \left( (-1)^{k + 1} \cdot \sum_{\substack{S \subseteq \cC \backslash \{C\}, \\ |S| = k}} |\cap_{C' \in S} P_{C'}| \right).
\end{equation}

Concretely, we think of $P$ as the set of all possible orders in which the bingo numbers can be called out,
and for a given collection $S$ of bingo cards other than $C$, the set $\cap_{C' \in S} P_{C'}$ consists of all orders in which $C$ loses to all of the bingo cards in $S$.

We can therefore enumerate the permutations $\sigma$ belonging to $\cap_{C' \in S} P_{C'}$ by:

\medskip

1. Setting $Y := (\cup_{C' \in S} C') \cup C$ and choosing a $|Y|$-element subset $X$ of $[n]$ with $\sigma(X)=Y$.

\medskip

2. Choosing an element $e \in C$ to be the last element from $Y$ to be called out in the order corresponding to $\sigma$.

\medskip

3. Choosing an ordering for $Y \backslash \{ e \}$.

\medskip

4. Choosing an ordering for $E \backslash Y$.

\medskip


This gives the formula:

\begin{align*}
|\cap_{C' \in S} P_{C'}| &= \binom{n}{|Y|} \cdot |C| \cdot (|Y| - 1)! \cdot (n - |Y|)! \\
&= n! \cdot \frac{|C|}{|(\cup_{C' \in S} C') \cup C|}.
\end{align*}

\medskip

Combining this with \eqref{eq:betaformula1} and \eqref{eq:inclusion-exclusion}, we obtain

\begin{align*}
\beta_C &= 1 - \sum_{k = 1}^{|\cC| - 1} \left( (-1)^{k + 1} \cdot \sum_{\substack{S \subseteq \cC \backslash \{C\}, \\ |S| = k}} \frac{|C|}{|(\cup_{C' \in S} C') \cup C|} \right) \\
&= \sum_{k = 0}^{|\cC| - 1} \left( (-1)^{k} \cdot \sum_{\substack{S \subseteq \cC \backslash \{C\}, \\ |S| = k}} \frac{|C|}{|(\cup_{C' \in S} C') \cup C|} \right).
\end{align*}

\end{proof}

\begin{thm} \label{thm:Tuttepolyformula}
Given a matroid $\cM$ on $E = [n]$ and a circuit $C$ of $\cM$, we have
\[\xB_C=\sum_{t=|C|}^{r(\cM)+1} \xB_C(t),\]
where
\[\xB_C(t) = \frac{|C|}{n} \cdot |\cI_{C, \, t - |C|}| \cdot \binom{n - 1}{t - 1}^{-1}.\]

Here $\cI_{C, k}$ denotes the collection of all $k$-element subsets $S$ of $E \backslash C$ such that $C$ is the only circuit contained in $S \cup C.$
Furthermore, for $0 \leq k \leq r(\cM / C)$ we have
$$|\cI_{C, \, k}| = |\{I \in \cI(\cM/C) : |I| = k \}| = [x^{r(\cM / C) - k}] T_{\cM/C}(x + 1, 1).$$
\end{thm}

\begin{proof}
For the first part, we prove, equivalently, that if $P_C(t)$ denotes the set of permutations of $E$ in which a given circuit $C$ is the first one to be completed, and this happens in round $t$, then
\[
|P_C(t)| = |C|\cdot |\cI_{C, \, t - |C|}| \cdot (t-1)!\cdot (n-t)!.
\]

We can enumerate all such permutations by:

\medskip

1. Choosing an element $e$ to be the last element of $C$ called out (in round \(t\)).

\medskip

2. Choosing a $(t-|C|)$-element subset $S \subseteq E \backslash C$ of additional elements to be called out before circuit $C$ wins, or in other words, an element of $\cI_{C, \, t - |C|}$.

\medskip

3. Choosing an ordering of the elements in $(C \cup S) \backslash \{e\}$.

\medskip

4. Choosing an ordering of $E \backslash (C \cup S)$.

\medskip

This gives the formula:
\[
|P_C(t)| = |C| \cdot |\cI_{c, \, t - |C|}| \cdot (t-1)! \cdot (n-t)!.
\]

And dividing by the $n!$ total number of permutations of $E$, we obtain:

$$\beta_C(t) = \frac{|P_C(t)|}{n!} = \frac{|C|}{n} \cdot |\cI_{C, \, t - |C|}| \cdot \binom{n - 1}{t - 1}^{-1}.$$

Since $\beta_C(t)$ is the probability of circuit $C$ winning on round $t$, the sum over $t$ gives the overall probability of $C$ winning.

\medskip

Note that setting $y = 1$ in \eqref{eq:tutte-def} shows that the number of independent sets of size \( k \) in a matroid $\cM$
is given by the coefficient of $x^{r(\cM) - k}$ in $T_{\cM}(x + 1,1)$. 

\medskip

{\bf Claim:} A subset $I$ of $E \backslash C$ is independent in $\cM / C$ iff $C$ is the only circuit of $\cM$ contained in $I \cup C$.

\medskip

Assuming the claim, we conclude that
$$|\cI_{C, \, k}| = |\{I \in \cI(\cM/C) : |I| = k \}| = [x^{r(\cM / C) - k)}] T_{\cM/C}(x + 1, 1)$$
as desired.

To prove the claim, choose $e \in C$ and let $J = C \backslash \{e\}$. It follows from \autoref{lem:basisrankcontract} that $I$
is independent in $\cM/J$ iff $I \cup J$ is independent in $\cM$. Moreover, since $\{ e \}$ is a loop of $\cM / C$, $I$ is independent in $\cM / C$ iff it is independent in $\cM / J$.

If $C$ is the only circuit of $\cM$ contained in $I \cup C$, then $I \cup J$ is independent in $\cM$ and so $I$ is independent in $\cM/C$.
Conversely, suppose $I$ is independent in $\cM / C$. If there is another circuit $C'$ contained in $I \cup C$, then since $I \cup J$ is independent in $\cM$, $C'$ must contain $e$. But then, by the circuit elimination axiom (C3), there exists a circuit $C''$ contained in the independent set $I \cup J$, a contradiction. Thus $C$ is the unique circuit contained in $I \cup C$.

\end{proof}

We now establish \autoref{thm:intro-logconcave}:

\begin{thm} \label{thm:logconcave}
Let \(\cM\) be a matroid on $[n]$ and let $C$ be a circuit of $\cM$. Then the sequence 
$\xB_C(1),\ \xB_C(2),\ldots,\ \xB_C(n)$
is log-concave.
\end{thm}

To prove \autoref{thm:logconcave} we will use (the strong version of) Mason's conjecture, proved independently in \cite[Theorem 4.14]{Branden-Huh20} and \cite[Theorem 1.2]{ultralogc}. For the statement,
    we say that a sequence $a_1,\ a_2,\ldots,\ a_n$ of real numbers is \textbf{ultra log-concave} if the sequence $\frac{a_1}{\binom{n}{1}},\ \frac{a_2}{\binom{n}{2}},\ldots,\ \frac{a_n}{\binom{n}{n}}$ is log concave \cite[p.316]{ligget}.

\begin{thm}[Strong version of Mason's conjecture] \label{thm:strongmason}
Let $\cM$ be a matroid on $[n]$, and let $\cI_k$ denote the collection of independent sets of $\cM$ of size $k$.
Then the sequence $|\cI_0|,|\cI_1|,\ldots,|\cI_n|$ is ultra log-concave.
\end{thm}

\begin{rem} \label{rem:no-internal-zeros}
Since subsets of an independent set are independent, it is easy to see that the non-negative sequence $|\cI_0|,|\cI_1|,\ldots,|\cI_n|$ has no internal zeros.
\end{rem}

\begin{proof}[Proof of \autoref{thm:logconcave}]
Recall the formula 
\[
\xB_C(t)=\frac{|C|}{n} \frac{|\cI_{C, \, t-|C|}|}{\binom{n-1}{t-1}},
\]
where $\cI_{C, \, k}$ is the collection of independent sets of $\cM / C$ of size $k$ and $|C|\leq t \leq n$.
Moreover, $\xB_C(t)=0$ for $0 \leq t < |C|$.

Since $n$ and $|C|$ are independent of $t$, to prove the theorem it suffices to prove that the sequence
$\frac{|\cI_{C, \, t-|C|}|}{\binom{n-1}{t-1}}$
with $t=|C|,|C|+1,\ldots,n$ is log-concave.

By \autoref{thm:strongmason}, applied to the matroid $\cM / C$ on $E \backslash C$, the sequence 
$\frac{|\cI_{C, \, k}|}{\binom{n-|C|}{k}}$
with $k=0,1,\ldots,r(\cM / C)-|C|$ is log-concave. (We don't need to consider values of $k$ larger than $r(\cM / C)$ because $|\cI_{C, \, k}|=0$ for such values.)

Setting $t=k+|C|$ shows that the sequence
\[
h(t) := \frac{|\cI_{C, \, t-|C|}|}{\binom{n-|C|}{t-|C|}}
\]
with $t=|C|,|C|+1,\ldots,r(\cM / C)$ is log-concave.


Now note that for $|C|\leq t \leq n$, we have
\[
\frac{|\cI_{C, \, t-|C|}|}{\binom{n-1}{t-1}}=\frac{|\cI_{C, \, t-|C|}|}{\binom{n-|C|}{t-|C|}}\cdot\frac{\binom{n-|C|}{t-|C|}}{\binom{n-1}{t-1}}.
\]

We claim that \(g(t) :=\frac{\binom{n-|C|}{t-|C|}}{\binom{n-1}{t-1}}\) is log-concave. Indeed,
\[\frac{g(t+1)}{g(t)}=\frac{\binom{n-|C|}{t+1-|C|}}{\binom{n-1}{t}}/\frac{\binom{n-|C|}{t-|C|}}{\binom{n-1}{t-1}}=\frac{t}{t+1-|C|}\] is non-increasing in \(t\) (its derivative is \(\frac{1-|C|}{(t+1-|C|)^2}\leq0\)), and so \[\frac{g(t)}{g(t-1)}\geq\frac{g(t+1)}{g(t)}\Leftrightarrow g(t)^2\geq g(t-1)g(t+1),\] i.e., \(g(t)\) is log-concave. 

We are now done, because both $f(t)$ and $g(t)$ are non-negative log-concave sequences, and by \cite{unimodal} the pointwise (Hadamard) product of two such sequences is again log-concave.(Proof: We have
\[
\begin{aligned}
\bigl(f(k)g(k)\bigr)^2
&= f(k)^2 g(k)^2 \\
&\ge f(k-1)f(k+1) \cdot g(k-1)g(k+1) \\
&= f(k-1)g(k-1)\cdot f(k+1)g(k+1)
\end{aligned}
\]
for all $k$.)
\end{proof}

\subsection{Upper and lower bounds on winning probabilities}



In some of the arguments that follow, we will use the ``falling factorial'' notation
\[
x^{\underline k} := x(x-1)\cdots (x-k+1).
\]

We now prove \autoref{Upper bound} and \autoref{intro-thmD}.

\begin{thm}\label{thmBC}
Let $\cM$ be a matroid of rank $r$ on $n$ elements and let $C$ be a circuit in \(\cM\). Then,
\[
\xB_C(t) \leq \frac{|C|}{n} \cdot \binom{n - |C|}{t - |C|} \cdot \binom{n - 1}{t - 1}^{-1}
\]
for all $1 \leq t \leq n$, and
\[
\xB_C \leq \frac{(r + 1)! \cdot (n - |C|)!}{n! \cdot (r+1-|C|)!} = \frac{\binom{r+1}{|C|}}{\binom{n}{|C|}}.
\]


\end{thm}

\begin{proof}
Note, from the proof of \autoref{thm:Tuttepolyformula}, that $|\cI_{C, k}|$ equals the number of $k$-element independent sets in $\cM/C$. The size of the ground set of $\cM / C$ is $n - |C|$, so an upper bound for the number of independent sets of size $k$ is $\binom{n - |C|}{k}$. Substituting this into the formula for $\beta_C(t)$, we obtain

$$\beta_C(t) \leq \frac{|C|}{n} \cdot \binom{n - |C|}{t - |C|} \cdot \binom{n - 1}{t - 1}^{-1}.$$

Since $\beta_C(t) = \sum_{t = |C|}^{r + 1} \beta_C(t)$, this gives

\begin{align*}
\beta_C &\leq \frac{|C|}{n} \sum_{t = |C|}^{r + 1} \frac{(n - |C|)!}{(t - |C|)! \cdot (n - t)!} \cdot \frac{(t - 1)! \cdot (n - t)!}{(n - 1)!} \\
&= \frac{|C| \cdot (n - |C|)!}{n!} \sum_{t = |C| - 1}^r t^{\underline{|C| - 1}} .
\end{align*}

Let $|C| = r + 1 - d$ with $0 \leq d < r + 1$.
By \cite[p.53]{ConcreteMathematics} (a discrete version of the fundamental theorem of calculus), for any integer \(c\neq -1\) we have $\sum_{k = a}^b k^{\underline{c}} = \frac{1}{c + 1} \left( (b + 1)^{\underline{c + 1}} - a^{\underline{c + 1}} \right)$, which gives:
\[
\begin{aligned}
\beta_C &\leq \frac{(n - |C|)!}{n!} \left( (r + 1)^{\underline{|C|}} - (|C| - 1)^{\underline{|C|}} \right) \\
&\leq \frac{(n - |C|)!}{n!} (r + 1)^{\underline{|C|}} \\
&= \frac{(r + 1)! \cdot (n - r + d - 1)!}{n! \cdot d!}.
\end{aligned}
\]
\end{proof}

\begin{rem}\label{rem:upper strict}
This upper bound is sharp, in the sense that given any nonnegative integers $n$, $r < n$, and $d \leq r$ we can construct a matroid on $E = [n]$ of rank $r$ with a circuit $C$ of size $r + 1 - d$ such that $\xB_C$ achieves this bound. Specifically, consider the matroid $U_{d, n - (r + 1 - d)} \oplus U_{r - d, r + 1 - d}$ and denote the unique circuit of the second summand by $C$. Note that appending any $r - d$ elements of $C$ to any basis of $U_{d, n - (r + 1 - d)}$ will yield a basis of $\cM$, so $r(\cM) = r$ as desired. Furthermore, the upper bound on $|\cI_{C, k}|$ is achieved, since $\cM/C = U_{d, n - (r + 1 - d)}$, which contains the desired numbers of independent sets, so $\xB_C$ achieves the desired upper bound as well.
\end{rem}

\begin{rem}
If $\cM$ contains coloops, then the bounds in Theorem~\ref{thmBC} can be improved by applying \autoref{coloops-don't-change-beta-values} and calculating the corresponding bound for $\cM/\cL^*$, where \(\cL^*\) is the set of coloops in $\cM$.
\end{rem}

The following is a strengthening of \autoref{intro-thmD}:

\begin{thm}\label{thmD}
Let $\cM$ be a matroid of rank $r$ on $n$ elements and let $C$ be a circuit in \(\cM\) with $|C| = r + 1 - d$ for some $0 \leq d < r + 1$. Then, if \(m\) denotes the number of coloops in \(\cM\), when \(n>r+1\),
$$\xB_C \geq \binom{n - 1 - d}{|C|}^{-1} - \frac{|C|}{n-m} \cdot \binom{n - m - 1}{r - m + 1}^{-1},$$ and when \(n=r+1\), \(\xB_C=1\).
\end{thm}

To prove \autoref{thmD}, we will need the following preliminary results.

\begin{lemma}\label{lem:combo identity}
If $a, b, s \geq 0$ are integers and $a + b < s$, then
$$\sum_{k = a}^{s - b - 1} \binom{k}{a} \binom{s - k - 1}{b} = \binom{s}{a + b + 1}.$$
%
%
\end{lemma}

\begin{proof}
By \cite[p.169]{ConcreteMathematics}, for integers $l, m \geq 0$ and $n \geq q \geq 0$ we have
$$\sum_{k = 0}^l \binom{l - k}{m} \binom{q + k}{n} = \binom{l + q + 1}{m + n + 1}.$$
Setting $q = a$, $n = a$, $m = b$, and $l = s - a - 1$ gives
$$\sum_{k = 0}^{s - a - 1} \binom{a + k}{a} \binom{(s - a - 1) - k}{b} = \binom{s}{a + b + 1}.$$
Re-indexing the sum and cutting off zero terms gives:
$$\sum_{k = a}^{s - b - 1} \binom{k}{a} \binom{s - k - 1}{b} = \binom{s}{a + b + 1}.$$
\end{proof}

\begin{lemma}\label{lem:basis count}
If $\cM$ is a coloopless matroid of rank $r$, then $\cM$ has at least $r+1$ distinct bases.
\end{lemma}

\begin{proof}
Let $B=\{b_1,\dots,b_r\}$ be a basis of $\cM$. Since $\cM$ has no coloops, for each $1 \leq i \leq r$, there exists a basis $B_i$ with $b_i \notin B_i$. Applying \autoref{lem:basisexchange} (the basis exchange property) to $b_i \in B \backslash B_i$, we find that for each $i$, there is an element $e_i \in B_i \backslash B$ such that $B'_i := (B \backslash \{b_i\}) \cup \{e_i\}$ is also a basis.

The bases $B'_1, \dots, B'_r$ are all distinct from one another, and from $B$. Indeed, if $i \neq j$, then $B'_i$ omits $b_i$ but contains $b_j$, while $B'_j$ omits $b_j$ but contains $b_i$, so $B'_i \neq B'_j$. Moreover, $B'_i \neq B$ since $b_i \notin B'_i$ but $b_i \in B$. Thus we have 
at least $r + 1$ distinct bases $B, B'_1, \dots, B'_r$.

\end{proof}




The following is a corollary of the Kruskal-Katona Theorem:

\begin{cor}\label{cor:Kruskal-Katona}  \cite[p.122]{anderson2002combinatorics}
Let $S$ be a collection of $r$-element sets, and for $k < r$ define 
\[
\Delta^{(k)} S := \{B : |B| = k, \ B \subseteq A \text{ for some } A \in S\}.
\]
If $|S| = \binom{a_r}{r} + \cdots + \binom{a_t}{t}$
for some $a_r > \cdots > a_t \geq t \geq 1$, then
$$|\Delta^{(k)} S| \geq \binom{a_r}{k} + \cdots + \binom{a_t}{k + t - r}.$$
\end{cor}

Combining \autoref{cor:Kruskal-Katona} and \autoref{lem:basis count}, we get the following lower bound on the number of independent sets in a coloopless matroid:

\begin{cor}\label{cor:independent set count}
Every coloopless rank $r$ matroid has at least $\binom{r + 1}{k}$ independent sets of cardinality $k$ for $0 \leq k \leq r$.
\end{cor}

\begin{proof}
By \autoref{lem:basis count}, there exists a set $S$ consisting of $r + 1$ distinct bases of $\cM$, each of which have size $r=r(\cM)$.
Since a subset of $E$ is independent if and only if it is contained in a basis, if $\cI_k$ denotes the collection of independent sets of $\cM$ of size $k$ then $\cI_k$ is equal to $\Delta^{(k)} S$ (using the notation from \autoref{cor:Kruskal-Katona}).
Setting $a_r = r + 1$ and $t = r$, \autoref{cor:Kruskal-Katona} tells us that $|\Delta^{(k)} S| \geq \binom{r + 1}{k}$, and hence $|\cI_k| \geq \binom{r + 1}{k}$ as desired.
\end{proof}

\begin{proof}[Proof of \autoref{thmD}]
Let \(\cL^*\) be the set of all coloops in \(\cM\), with \(|\cL^*|=m\), and consider the contraction \(\cM'=(E', \cC')=\cM / \cL^*\), where \(n'=|E'|=n-m\) and \(r'=r(\cM')=r(\cM)-m\) by \autoref{lem:basisrankcontract}. 
Denote by \(C'\) the circuit \(C\backslash \cL^*\), which has size \(|C'|=|C|=r'+1-d'\), where \(d'=d-m\). The matroid \(\cM'\) is coloopless,
and by \autoref{lem:coloopless},
$\cM'' := \cM'/C'$ is also coloopless.

By \autoref{cor:contracting-a-circuit} and \autoref{cor:independent set count}, we conclude that $\cM''$ has at least $\binom{d' + 1}{k}$ independent sets of cardinality $k$. 
Therefore,
$$\beta_{C'}(t) \geq \frac{|C'|}{n'} \cdot \binom{d' + 1}{t - |C'|} \cdot \binom{n' - 1}{t - 1}^{-1},$$
and summing over all \(t\) we get:
$$\beta_{C'} \geq \frac{|C'|}{n'} \sum_{t = |C'|}^{r' + 1} \binom{d' + 1}{t - |C'|} \cdot \binom{n' - 1}{t - 1}^{-1}.$$

\medskip

\textbf{Case 1:} \( n'=r'+1. \)

\medskip

By \autoref{lem:basis count}, $\cM'$ has at least \(r'+1\) bases (of size \(r'\)). However, there are only \(\binom{n'}{r'}=r'+1\) subsets of $E'$ in \(\cM'\) of size \(r'\). Therefore $\cM'$ has a unique circuit, namely the ground set $E'$, so $|C'|=n'$. 
Thus $d'=0$ and
$$\beta_{C'} \geq \frac{|C'|}{n'} \sum_{t = |C'|}^{r' + 1} \binom{1}{t - |C'|} \cdot \binom{r'}{t - 1}^{-1}.$$

This means $|C'| = r' + 1$ which collapses this sum to one term, and \(\xB_{C'}\) can at most be \(1\), so

\[1\geq\beta_{C'} \geq \frac{n'}{n'} \binom{1}{0}\binom{r'}{r'}^{-1}=1\ \Longrightarrow\ \xB_{C'}=1. \]

\medskip

\textbf{Case 2:} \( n'>r'+1. \)

\medskip

In this case, standard algebraic manipulations give:
\[
\begin{aligned}
\beta_{C'} &\geq \frac{|C'|}{n'} \sum_{t = |C'|}^{r' + 1} \binom{d' + 1}{t - |C'|} \cdot \binom{n' - 1}{t - 1}^{-1} \\
&\geq \frac{|C'|}{n'} \sum_{t = |C|}^{r' + 1} \frac{(d' + 1)!}{(t - |C'|)! \cdot (r' + 2 - t)!} \cdot \frac{(t - 1)! \cdot (n' - t)!}{(n' - 1)!} \\
&\geq \frac{(d' + 1)! \cdot |C'|! \cdot (n' - r' - 2)!}{n'!} \sum_{t = |C'|}^{r' + 1} \binom{t - 1}{|C'| - 1} \binom{n' - t}{n' - r' - 2} \\
&= \frac{(d' + 1)! \cdot |C'|! \cdot (n' - r' - 2)!}{n'!} \sum_{t = |C'| - 1}^{r'} \binom{t}{|C'| - 1} \binom{n' - t - 1}{n' - r' - 2}.
\end{aligned}
\]


Applying \autoref{lem:combo identity} and simplifying gives:
$$\beta_{C'} \geq \binom{n' - 1 - d'}{r'  + 1 - d'}^{-1} - \frac{r' + 1 - d'}{n'} \cdot \binom{n' - 1}{r' + 1}^{-1}.$$

To finish the proof, note that, by \autoref{coloops-don't-change-beta-values}, we have \(\xB_{C'}=\xB_C\). Since $n = n' + m$ and $r = r' + m$, 
we have
$n' = r' + 1$ iff $n = r + 1$, and in this case $\beta_C = \beta_{C'} = 1$. 
Similarly, $n' > r' + 1$ iff $n > r + 1$, and in this case
\begin{align*}
    \xB_C=\xB_{C'}&\geq \binom{n' - 1 - d'}{r'  + 1 - d'}^{-1} - \frac{r' + 1 - d'}{n'} \cdot \binom{n' - 1}{r' + 1}^{-1} \\&= \binom{n - 1 - d}{|C|}^{-1} - \frac{|C|}{n-m} \cdot \binom{n - m - 1}{r - m + 1}^{-1}.
\end{align*}
\end{proof}

\begin{rem}\label{rem:lower strict}
The lower bound in \autoref{thmD} is sharp, in the sense that given any nonnegative integers $n$, $r < n-1$, and $d \leq r$, we can construct a matroid on $E = [n]$ of rank $r$ with a circuit $C$ of size $r + 1 - d$ such that $\xB_C$ achieves this bound. Specifically, consider the matroid \[U_{d, d+1} \oplus U_{r-d, r+1-d} \oplus \underbrace{U_{0,1} \oplus \cdots \oplus U_{0,1}}_{q\ \text{loops}},\] 
where \(q=n-r-2\), and denote the unique circuit of \(U_{r-d, r+1-d}\) by $C$. Then the lower bound on $|\cI_{C, k}|$ is achieved, since $\cM/C = U_{d, d+1} \bigoplus_{i \in [q]} U_{0,1}$ contains the desired numbers of independent sets, so $\xB_C$ achieves the lower bound as well.
\end{rem}

The following simplified version of \autoref{thmD}, which is \autoref{intro-thmD}, has a more pleasant appearance (although it is no longer sharp when $d>0$):

\begin{thm}\label{cor:thmD-simplified}
Let $\cM$ be a matroid of rank $r$ on $n$ elements and let $C$ be a circuit in \(\cM\) with $|C| = r + 1 - d$ for some $0 \leq d < r + 1$. Then 
$$\xB_C \geq \binom{n - d}{|C|}^{-1}.$$
\end{thm}

\begin{proof}
Let \(\cL^*\) be the set of all coloops in \(\cM\), with \(|\cL^*|=m\), and consider the contraction \(\cM'=(E', \cC')=\cM / \cL^*\), where \(n'=|E'|=n-m\) and \(r'=r(\cM')=r(\cM)-m\) by \autoref{lem:basisrankcontract}. 
Denote by \(C'\) the restriction of $C$ to \(E\backslash \cL^*\), which has size \(|C'|=|C|=r'+1-d'\), where \(d'=d-m\).

In \autoref{thmD}, \(|\cI_{C',t-|C'|}|\) is bounded from below by \(\binom{d'+1}{t-|C'|}\), and therefore the weaker bound
\[ |\cI_{C',t-|C'|}| \geq \binom{d'}{t-|C'|} \]
also holds.
Now, following the proof of \autoref{thmD}, standard algebraic manipulations give:
\[
\begin{aligned}
\beta_{C'} &\geq \frac{|C'|}{n'} \sum_{t = |C'|}^{r' + 1} \binom{d'}{t - |C'|} \cdot \binom{n' - 1}{t - 1}^{-1} \\
&\geq \frac{|C'|}{n'} \sum_{t = |C|}^{r' + 1} \frac{(d' )!}{(t - |C'|)! \cdot (r' + 1 - t)!} \cdot \frac{(t - 1)! \cdot (n' - t)!}{(n' - 1)!} \\
&\geq \frac{(d')! \cdot |C'|! \cdot (n' - r' - 1)!}{n'!} \sum_{t = |C'|}^{r' + 1} \binom{t - 1}{|C'| - 1} \binom{n' - t}{n' - r' - 1} \\
&= \frac{(d')! \cdot |C'|! \cdot (n' - r' - 1)!}{n'!} \sum_{t = |C'| - 1}^{r'} \binom{t}{|C'| - 1} \binom{n' - t - 1}{n' - r' - 1} \\
&= \frac{(d')! \cdot |C'|! \cdot (n' - r' - 1)!}{n'!} \binom{n'}{n' - d'}, 
\end{aligned}
\]
where the final equality holds by \autoref{lem:combo identity}.
Simplifying and applying \autoref{coloops-don't-change-beta-values} gives the desired result:

$$\xB_C=\xB_{C'} \geq \frac{|C'|! \cdot (n' - r' - 1)!}{(n' - d')!} = \binom{n' - d'}{|C'|}^{-1}=\binom{n-d}{|C|}^{-1}.$$
\end{proof}

As an immediate corollary (with $d=0$) of \autoref{thmBC} and \autoref{cor:thmD-simplified}, we obtain:
\begin{cor} \label{cor:r+1beta}
If $|C|=r+1$, then $\xB_C=\binom{n}{r+1}^{-1}.$
\end{cor}

One can also give a direct proof of Corollary~\ref{cor:r+1beta}, bypassing the lengthy computations appearing in the proofs of \autoref{thmBC} and \autoref{cor:thmD-simplified}. Since the argument is short and pleasant, we present it here. 

\begin{proof}[Direct proof of Corollary~\ref{cor:r+1beta}]
If a circuit $C$ is to win the game of matroid bingo, the victory must occur in round $r + 1$, with the first $r$ numbers called out forming a basis. There are $r + 1$ such bases $\{B_1, ..., B_{r + 1}\}$ which can occur, corresponding to the different $r$-element subsets of $C$. Let $E_i$ denote the event that the first $r$ numbers called out form the set $B_i$. By Bayes' theorem, we have
\[
\mathbb{P}(C \textrm{ wins}) = \sum_{i = 1}^{r + 1} \mathbb{P}(C \textrm{ wins} \; | \; E_i) \cdot \mathbb{P}(E_i).
\]

On the other hand, we have $\mathbb{P}(E_i) = \binom{n}{r}^{-1}$, since there are $\binom{n}{r}$ possible choices for the set consisting of the first $r$ numbers called out, and each is equally likely to occur. Furthermore, assuming event $E_i$ occurs, each of the $n - r$ possible numbers which can be called out in round $r + 1$ will lead to a distinct circuit winning, and precisely one of these is the unique element of $C-B_i$. Again, each of these outcomes is equally likely, so $\mathbb{P}(C \textrm{ wins} \; | \; E_i) = \frac{1}{n - r}$. Therefore
\[
\begin{aligned}
\mathbb{P}(C \textrm{ wins}) &= \sum_{i = 1}^{r + 1} \frac{1}{n - r} \binom{n}{r}^{-1} = \frac{r + 1}{n - r} \binom{n}{r}^{-1} \\
&= \frac{(r + 1) \cdot r! \cdot (n - r)!}{(n - r) \cdot n!} = \frac{(r + 1)! \cdot (n - r - 1)!}{n!} = \binom{n}{r + 1}^{-1}.
\end{aligned}
\]
\end{proof}

%

\subsection{Monotonicity theorems}

The following consequence of \autoref{thmBC} 
shows that there can be no monotonicity violations involving a circuit of maximum size $r+1$.

\begin{cor}\label{corA}
Let $\cM$ be a matroid of rank $r$ on $n$ elements. Let $C$ and $C'$ be distinct circuits of $\cM$ with $|C|=r+1$ and $|C'|=r+1-d$ with $0\leq d<r+1$. Then $$\xB_C \leq \xB_{C'},$$ with equality if and only if $d=0$.
\end{cor}

\begin{proof}
By \autoref{cor:r+1beta}, we know that $\xB_C = \binom{n}{r + 1}^{-1}$, and from \autoref{cor:thmD-simplified} we know that 
$\xB_{C'} \geq \binom{n - d}{r + 1 - d}^{-1}$.
So it suffices to prove that
\begin{equation} \label{eq:yet-another-binomial-inequality}
\binom{n}{r+1}^{-1} \;\le\; \binom{n-d}{\,r+1-d\,}^{-1}
\end{equation}
for all integers $n,r,d$ with $0 \le d < r+1 \le n$.

To prove \eqref{eq:yet-another-binomial-inequality}, consider the ratio
\[
\frac{\binom{n-d}{\,r+1-d\,}}{\binom{n}{r+1}}
= \frac{(n-d)!\,(r+1)!}{(r+1-d)!\,n!} = \frac{(r+1)^{\underline{d}}}{n^{\underline{d}}}.
\]
Since $n \ge r+1$, each factor in $n^{\underline{d}}$ is at least as large
as the corresponding factor in $(r+1)^{\underline{d}}$, so
\[
\frac{(r+1)^{\underline{d}}}{n^{\underline{d}}} \;\le\; 1.
\]
Hence $\binom{n-d}{r+1-d} \le \binom{n}{r+1}$, and taking reciprocals yields \eqref{eq:yet-another-binomial-inequality} 
(with equality iff $d=0$ or $n=r+1$). 

In our application, we cannot have $n=r+1$, because that would mean that $\cM$ has a unique circuit, so $\xB_C = \xB_{C'}$ only when $d=0$.
\end{proof}

Using the bounds given by \autoref{thmBC} and \autoref{cor:thmD-simplified}, we show that monotonicity violations cannot occur past a certain threshold (\autoref{thm:intro-threshhold}). 

\begin{thm} \label{thm:threshhold}
For fixed $r$, there exists $N = N(r)$ so that if $\cM$ is a matroid of rank $r$ on $n \geq N$ elements, then 
monotonicity holds for $\cM$.
\end{thm}


\begin{proof}
Let $\cM$ be a matroid of rank $r$ on $n$ elements, and let $C$ and $C'$ be any pair of circuits in $\cM$ with $|C| > |C'|$. Write $|C| = r + 1 - d$ and $|C'| = r + 1 - d'$, so that $d < d'$. We wish to show that $\xB_C < \xB_{C'}$. We will in fact show the stronger result that
the upper bound $U_n$ on $\xB_C$ given by \autoref{thmBC} is smaller than the lower bound $L_n$ on $\xB_{C'}$ given by \autoref{cor:thmD-simplified}, i.e., that
$$U_n := \frac{(r+1)!(n-r-1+d)!}{n!d!} < L_n := \frac{(r+1-d')!(n-r-1)!}{(n-d')!}$$
for $n \gg 0$.

Writing $C_1 = \frac{(r + 1)!}{d!}$, we have

\[
\begin{aligned}
    U_n &= \frac{(r+1)!(n-r-1+d)!}{n!d!} \\ &= C_1\frac{(n-r-1+d)!}{n!}\\
    &=C_1\prod_{i=0}^{r-d}\frac{1}{n-i}\sim C_1\frac{1}{n^{r+1-d}},
\end{aligned}
\]
where $f(n) \sim g(n)$ means that $\lim_{n\to\infty} f(n)/g(n)=1$.

Similarly, writing $C_2 = (r + 1 - d')!$, we have
\[
\begin{aligned}
    L_n &= \frac{(r+1-d')!(n-r-1)!}{(n-d')!} \\ &= C_2\frac{(n-r-1)!}{(n-d')!}\\
    &= C_2\prod_{i=0}^{r-d'}\frac{1}{n-d'-i}
    \sim C_2\frac{1}{n^{r+1-d'}}.\\
\end{aligned}
\]

Thus,
$$\frac{U_n}{L_n} \sim \frac{C_1\frac{1}{n^{r+1-d}}}{C_2\frac{1}{n^{r+1-d'}}} = Cn^{d-d'},$$
where $C = \frac{C_1}{C_2}=\frac{(r+1)!}{d!(r+1-d')!}$. 
Since $d < d'$, it follows that 
$\lim_{n \to \infty} \frac{U_n}{L_n} = 0$, and in particular
there is a threshold $N'$ (depending on $r, d$ and $d'$) such that $U_n < L_n$ for $n \geq N'$.

Finally, if $N$ is the maximum of all these thresholds as $d$ and $d'$ vary through all elements of $\{ 0,1,\ldots,r \}$, 
we conclude that for $n \geq N$ we have $\xB_C < \xB_{C'}$ for all pairs $C,C'$ of circuits of any matroid of rank $r$ on $[n]$ such that $|C|>|C'|$.
\end{proof}


\begin{rem}
Neither the monotonicity property nor its negation is preserved under taking minors.
For the first statement, if we add a loop to the matroid $\cM$ in \autoref{tab:matroid-1}, which has a monotonicity violation, the resulting matroid $\cM \oplus U_{0,1}$ (of which $\cM$ is a minor) does not appear in \autoref{Appendix A} and therefore satisfies monotonicity. 
For the second statement, note that any proper minor of $\cM$ has at most 7 elements and therefore satisfies monotonicity.

Similarly, neither the monotonicity property nor its negation is preserved under taking direct sums.
The first statement follows from the fact that the same matroid $\cM$ from the previous paragraph, which violates monotonicity, is the direct sum of two matroids on 4 elements, both of which satisfy monotonicity. 
For the second statement, we checked (by computer) that the matroid $\cM \oplus \cM$ satisfies monotonicity. 
\end{rem}

\subsection{Equitable matroids}

Recall that a matroid $\cM$ is \emph{equitable} if all circuits $C$ yield the same value of $\xB_C$.
It is clear, by symmetry considerations, that uniform matroids are equitable.

However, there are non-uniform matroids which are also equitable. For example, the \textbf{dual Fano matroid} $F_7^*$, which is a rank 4 matroid on 7 elements, has this property:

\begin{center}
\begin{tabular}{| c | c |}
    \hline
    \multicolumn{2}{|c|}{\textbf{Winning probabilities for the circuits of $F_7^*$}} \\
    \hline
    $C$ & $\xB_C$\\
    \hline
     1245 & $1/7$\\ 
    \hline
     1237 & $1/7$\\ 
    \hline
     1356 & $1/7$\\
    \hline
     1467 & $1/7$\\ 
    \hline
     2346 & $1/7$\\
    \hline
     2567 & $1/7$\\
    \hline
     3457 & $1/7$\\
    \hline
\end{tabular}
\end{center}

\medskip

The dual Fano matroid is an example of a dual projective geometry over a finite field.
More generally:

\begin{df*}
Let $q$ be a prime power and let $n$ be a nonnegative integer. Denote by $\mathrm{PG}(n,q)$ the \textbf{projective geometry matroid} of rank $n+1$ over the finite field $\mathrm{GF}(q)$ with $q$ elements. Its ground set $E$ is the set of 1-dimensional subspaces of the $(n+1)$-dimensional vector space $\mathrm{GF}(q)^{n+1}$, and a subset of $E$ is independent if and only if the corresponding 1-dimensional subspaces are linearly independent in $\mathrm{GF}(q)^{n+1}$. 
\end{df*}

We prove the following result (\autoref{thm:intro-equitable}):

\begin{thm} \label{thm:equitable}
\begin{enumerate}
    \item[]
    \item If $C,C'$ are circuits of a matroid $\cM$ and there is an automorphism $\varphi$ of $\cM$ with $\varphi(C)=C'$, then $\xB_{C} = \xB_{C'}$. In particular, if the automorphism group of $\cM$ acts transitively on circuits, then $\cM$ is equitable.
    \item There exists an equitable matroid $\cM$ whose automorphism group does not act transitively on the set of circuits.
    \item Uniform matroids and duals of projective geometries over finite fields are equitable.
\end{enumerate}
\end{thm}

\begin{proof}
For (1), circuit $C$ wins when the bingo caller chooses numbers according to the permutation $\sigma$ of $[n]$ iff $C'$ wins when the 
bingo caller chooses numbers according to the permutation $\phi \circ \sigma$. Since left-multiplication by $\phi$ permutes the elements of $S_n$, the result follows immediately.

For (2), consider the matroids $\cM_1 = U_{4,6}$ on $\{1, ..., 6\}$, $\cM_2 = U_{8,9}$ on $\{7, ..., 15\}$, and their direct sum $\cM = \cM_1 \oplus \cM_2$. 
A straightforward calculation using \autoref{thm:mainbetaformula} shows that for every circuit $C$ in this matroid, $\xB_C = \frac{1}{7}$, so $\cM$ is indeed equitable. And since automorphisms act bijectively on the ground set, if $C$ and $C'$ are circuits of a matroid with $|C| \neq |C'|$, there cannot be an automorphism $\varphi$ with $\varphi(C) = C'$. As the circuits of $\cM_1$ all have size 5 and the single circuit of $\cM_2$ has size 9, the automorphism group cannot act transitively on the circuits of $\cM$.

For (3), if $\cM = U_{r,n}$ the circuits of $\cM$ are all subsets of rank $r+1$. It is well-known, and easy to see, that (a) the symmetric group $S_n$ acts by automorphisms on $U_{r,n}$, and (b) $S_n$ acts transitively on the set of subsets of $[n]$ of size $k$ for all $1 \leq k \leq n$. (Proof: Let $C,C’\subseteq\{1,\dots,n\}$ with $|C|=|C’|=k$. Choose any bijection
$f:C\to C’$ and any bijection $g: [n]\setminus C \to [n]\setminus C’$.
Define $\sigma(i)= f(i)$ for $i \in C$ and $\sigma(i)=g(i)$ for $i\not\in C$. Then $\sigma(C)=C'$.)
The result now follows from (1).

If $\cM= \mathrm{PG}(n,q)$, it is well-known that the automorphism group of $\mathrm{PG}(n,q)$ acts transitively on the set of hyperplanes (see, e.g.,~\cite[p.18]{nelson2011exponentially}). Since circuits of a matroid $\cM$ are just complements of hyperplanes of the dual matroid $\cM^*$ \cite[Proposition 2.1.22]{oxley2011} and $\cM$ and $\cM^*$ have the same automorphism group \cite[Proposition 2.1.21]{oxley2011}, the result follows from (1).
\end{proof}

\appendix

\section{Monotonicity violations in matroids with $|E|\leq 9$}
\label{Appendix A}




The following matroids are, up to isomorphism, the only matroids having up to 9 elements in which there is a monotonicity violation.

In particular, monotonicity holds for all matroids on at most 7 elements, and the matroid in \autoref{tab:matroid-1} is the unique matroid of size 8 in which there is a monotonicity violation.

\begin{longtable}{@{}l c r @{\,$\approx$\,} l@{}}
\caption{Matroid (n=8, r=5).}
\label{tab:matroid-1}\\
\toprule
\textbf{Circuit} & $\mathbf{|C|}$ & \multicolumn{2}{c}{$\boldsymbol{\xB}$-values (fraction \& decimal)} \\
\midrule
\endfirsthead
\toprule
\textbf{Circuit} & $\mathbf{|C|}$ & \multicolumn{2}{c}{$\boldsymbol{\xB}$-values (fraction \& decimal)} \\
\midrule
\endhead
\bottomrule
\endfoot
$\{5, 6, 7, 8\}$ & 4 & $3/14$ & 0.2143 \\
$\{1, 2, 3\}$ & 3 & $11/56$ & 0.1964 \\
$\{1, 2, 4\}$ & 3 & $11/56$ & 0.1964 \\
$\{1, 3, 4\}$ & 3 & $11/56$ & 0.1964 \\
$\{2, 3, 4\}$ & 3 & $11/56$ & 0.1964 \\
\end{longtable}

\begin{longtable}{@{}l c r @{\,$\approx$\,} l@{}}
\caption{Matroid (n=9, r=5).}
\label{tab:matroid-2}\\
\toprule
\textbf{Circuit} & $\mathbf{|C|}$ & \multicolumn{2}{c}{$\boldsymbol{\xB}$-values (fraction \& decimal)} \\
\midrule
\endfirsthead
\toprule
\textbf{Circuit} & $\mathbf{|C|}$ & \multicolumn{2}{c}{$\boldsymbol{\xB}$-values (fraction \& decimal)} \\
\midrule
\endhead
\bottomrule
\endfoot
$\{1, 2\}$ & 2 & $49/180$ & 0.2722 \\
$\{3, 7\}$ & 2 & $49/180$ & 0.2722 \\
$\{4, 5, 6, 8\}$ & 4 & $13/126$ & 0.1032 \\
$\{1, 3, 9\}$ & 3 & $37/420$ & 0.0881 \\
$\{2, 3, 9\}$ & 3 & $37/420$ & 0.0881 \\
$\{1, 7, 9\}$ & 3 & $37/420$ & 0.0881 \\
$\{2, 7, 9\}$ & 3 & $37/420$ & 0.0881 \\
\end{longtable}

\begin{longtable}{@{}l c r @{\,$\approx$\,} l@{}}
\caption{Matroid (n=9, r=5).}
\label{tab:matroid-3}\\
\toprule
\textbf{Circuit} & $\mathbf{|C|}$ & \multicolumn{2}{c}{$\boldsymbol{\xB}$-values (fraction \& decimal)} \\
\midrule
\endfirsthead
\toprule
\textbf{Circuit} & $\mathbf{|C|}$ & \multicolumn{2}{c}{$\boldsymbol{\xB}$-values (fraction \& decimal)} \\
\midrule
\endhead
\bottomrule
\endfoot
$\{1, 2\}$ & 2 & $269/1260$ & 0.2135 \\
$\{1, 9\}$ & 2 & $269/1260$ & 0.2135 \\
$\{2, 9\}$ & 2 & $269/1260$ & 0.2135 \\
$\{3, 4, 6, 8\}$ & 4 & $2/21$ & 0.0952 \\
$\{1, 5, 7\}$ & 3 & $37/420$ & 0.0881 \\
$\{2, 5, 7\}$ & 3 & $37/420$ & 0.0881 \\
$\{5, 7, 9\}$ & 3 & $37/420$ & 0.0881 \\
\end{longtable}

\begin{longtable}{@{}l c r @{\,$\approx$\,} l@{}}
\caption{Matroid (n=9, r=5).}
\label{tab:matroid-4}\\
\toprule
\textbf{Circuit} & $\mathbf{|C|}$ & \multicolumn{2}{c}{$\boldsymbol{\xB}$-values (fraction \& decimal)} \\
\midrule
\endfirsthead
\toprule
\textbf{Circuit} & $\mathbf{|C|}$ & \multicolumn{2}{c}{$\boldsymbol{\xB}$-values (fraction \& decimal)} \\
\midrule
\endhead
\bottomrule
\endfoot
$\{1, 2\}$ & 2 & $49/180$ & 0.2722 \\
$\{3, 4, 6, 9\}$ & 4 & $1/9$ & 0.1111 \\
$\{1, 5, 7\}$ & 3 & $37/420$ & 0.0881 \\
$\{2, 5, 7\}$ & 3 & $37/420$ & 0.0881 \\
$\{1, 5, 8\}$ & 3 & $37/420$ & 0.0881 \\
$\{2, 5, 8\}$ & 3 & $37/420$ & 0.0881 \\
$\{1, 7, 8\}$ & 3 & $37/420$ & 0.0881 \\
$\{2, 7, 8\}$ & 3 & $37/420$ & 0.0881 \\
$\{5, 7, 8\}$ & 3 & $37/420$ & 0.0881 \\
\end{longtable}

\begin{longtable}{@{}l c r @{\,$\approx$\,} l@{}}
\caption{Matroid (n=9, r=5).}
\label{tab:matroid-5}\\
\toprule
\textbf{Circuit} & $\mathbf{|C|}$ & \multicolumn{2}{c}{$\boldsymbol{\xB}$-values (fraction \& decimal)} \\
\midrule
\endfirsthead
\toprule
\textbf{Circuit} & $\mathbf{|C|}$ & \multicolumn{2}{c}{$\boldsymbol{\xB}$-values (fraction \& decimal)} \\
\midrule
\endhead
\bottomrule
\endfoot
$\{1, 4, 5, 7\}$ & 4 & $5/42$ & 0.1190 \\
$\{2, 3, 6\}$ & 3 & $37/420$ & 0.0881 \\
$\{2, 3, 8\}$ & 3 & $37/420$ & 0.0881 \\
$\{2, 6, 8\}$ & 3 & $37/420$ & 0.0881 \\
$\{3, 6, 8\}$ & 3 & $37/420$ & 0.0881 \\
$\{2, 3, 9\}$ & 3 & $37/420$ & 0.0881 \\
$\{2, 6, 9\}$ & 3 & $37/420$ & 0.0881 \\
$\{3, 6, 9\}$ & 3 & $37/420$ & 0.0881 \\
$\{2, 8, 9\}$ & 3 & $37/420$ & 0.0881 \\
$\{3, 8, 9\}$ & 3 & $37/420$ & 0.0881 \\
$\{6, 8, 9\}$ & 3 & $37/420$ & 0.0881 \\
\end{longtable}

\begin{longtable}{@{}l c r @{\,$\approx$\,} l@{}}
\caption{Matroid (n=9, r=6).}
\label{tab:matroid-6}\\
\toprule
\textbf{Circuit} & $\mathbf{|C|}$ & \multicolumn{2}{c}{$\boldsymbol{\xB}$-values (fraction \& decimal)} \\
\midrule
\endfirsthead
\toprule
\textbf{Circuit} & $\mathbf{|C|}$ & \multicolumn{2}{c}{$\boldsymbol{\xB}$-values (fraction \& decimal)} \\
\midrule
\endhead
\bottomrule
\endfoot
$\{1, 9\}$ & 2 & $641/1260$ & 0.5087 \\
$\{2, 3, 5, 6, 8\}$ & 5 & $37/252$ & 0.1468 \\
$\{1, 2, 4, 7\}$ & 4 & $13/90$ & 0.1444 \\
$\{2, 4, 7, 9\}$ & 4 & $13/90$ & 0.1444 \\
$\{1, 3, 4, 5, 6, 7, 8\}$ & 7 & $1/36$ & 0.0278 \\
$\{3, 4, 5, 6, 7, 8, 9\}$ & 7 & $1/36$ & 0.0278 \\
\end{longtable}

\begin{longtable}{@{}l c r @{\,$\approx$\,} l@{}}
\caption{Matroid (n=9, r=6).}
\label{tab:matroid-7}\\
\toprule
\textbf{Circuit} & $\mathbf{|C|}$ & \multicolumn{2}{c}{$\boldsymbol{\xB}$-values (fraction \& decimal)} \\
\midrule
\endfirsthead
\toprule
\textbf{Circuit} & $\mathbf{|C|}$ & \multicolumn{2}{c}{$\boldsymbol{\xB}$-values (fraction \& decimal)} \\
\midrule
\endhead
\bottomrule
\endfoot
$\{3, 4, 6, 8\}$ & 4 & $3/14$ & 0.2143 \\
$\{1, 2, 7\}$ & 3 & $11/56$ & 0.1964 \\
$\{1, 2, 9\}$ & 3 & $11/56$ & 0.1964 \\
$\{1, 7, 9\}$ & 3 & $11/56$ & 0.1964 \\
$\{2, 7, 9\}$ & 3 & $11/56$ & 0.1964 \\
\end{longtable}

\begin{longtable}{@{}l c r @{\,$\approx$\,} l@{}}
\caption{Matroid (n=9, r=6).}
\label{tab:matroid-8}\\
\toprule
\textbf{Circuit} & $\mathbf{|C|}$ & \multicolumn{2}{c}{$\boldsymbol{\xB}$-values (fraction \& decimal)} \\
\midrule
\endfirsthead
\toprule
\textbf{Circuit} & $\mathbf{|C|}$ & \multicolumn{2}{c}{$\boldsymbol{\xB}$-values (fraction \& decimal)} \\
\midrule
\endhead
\bottomrule
\endfoot
$\{1, 2, 7\}$ & 3 & $19/60$ & 0.3167 \\
$\{1, 5, 9\}$ & 3 & $32/105$ & 0.3048 \\
$\{3, 4, 6, 8, 9\}$ & 5 & $37/252$ & 0.1468 \\
$\{2, 5, 7, 9\}$ & 4 & $13/90$ & 0.1444 \\
$\{1, 3, 4, 5, 6, 8\}$ & 6 & $5/84$ & 0.0595 \\
$\{2, 3, 4, 5, 6, 7, 8\}$ & 7 & $1/36$ & 0.0278 \\
\end{longtable}

\begin{longtable}{@{}l c r @{\,$\approx$\,} l@{}}
\caption{Matroid (n=9, r=6).}
\label{tab:matroid-9}\\
\toprule
\textbf{Circuit} & $\mathbf{|C|}$ & \multicolumn{2}{c}{$\boldsymbol{\xB}$-values (fraction \& decimal)} \\
\midrule
\endfirsthead
\toprule
\textbf{Circuit} & $\mathbf{|C|}$ & \multicolumn{2}{c}{$\boldsymbol{\xB}$-values (fraction \& decimal)} \\
\midrule
\endhead
\bottomrule
\endfoot
$\{1, 2, 7\}$ & 3 & $19/60$ & 0.3167 \\
$\{3, 4, 5, 6, 9\}$ & 5 & $1/6$ & 0.1667 \\
$\{1, 2, 8, 9\}$ & 4 & $13/90$ & 0.1444 \\
$\{1, 7, 8, 9\}$ & 4 & $13/90$ & 0.1444 \\
$\{2, 7, 8, 9\}$ & 4 & $13/90$ & 0.1444 \\
$\{1, 2, 3, 4, 5, 6, 8\}$ & 7 & $1/36$ & 0.0278 \\
$\{1, 3, 4, 5, 6, 7, 8\}$ & 7 & $1/36$ & 0.0278 \\
$\{2, 3, 4, 5, 6, 7, 8\}$ & 7 & $1/36$ & 0.0278 \\
\end{longtable}

\begin{longtable}{@{}l c r @{\,$\approx$\,} l@{}}
\caption{Matroid (n=9, r=6).}
\label{tab:matroid-10}\\
\toprule
\textbf{Circuit} & $\mathbf{|C|}$ & \multicolumn{2}{c}{$\boldsymbol{\xB}$-values (fraction \& decimal)} \\
\midrule
\endfirsthead
\toprule
\textbf{Circuit} & $\mathbf{|C|}$ & \multicolumn{2}{c}{$\boldsymbol{\xB}$-values (fraction \& decimal)} \\
\midrule
\endhead
\bottomrule
\endfoot
$\{2, 4, 9\}$ & 3 & $32/105$ & 0.3048 \\
$\{3, 5, 6, 8, 9\}$ & 5 & $37/252$ & 0.1468 \\
$\{1, 2, 4, 7\}$ & 4 & $13/90$ & 0.1444 \\
$\{1, 2, 7, 9\}$ & 4 & $13/90$ & 0.1444 \\
$\{1, 4, 7, 9\}$ & 4 & $13/90$ & 0.1444 \\
$\{2, 3, 4, 5, 6, 8\}$ & 6 & $5/84$ & 0.0595 \\
$\{1, 2, 3, 5, 6, 7, 8\}$ & 7 & $1/36$ & 0.0278 \\
$\{1, 3, 4, 5, 6, 7, 8\}$ & 7 & $1/36$ & 0.0278 \\
\end{longtable}

\begin{longtable}{@{}l c r @{\,$\approx$\,} l@{}}
\caption{Matroid (n=9, r=6).}
\label{tab:matroid-11}\\
\toprule
\textbf{Circuit} & $\mathbf{|C|}$ & \multicolumn{2}{c}{$\boldsymbol{\xB}$-values (fraction \& decimal)} \\
\midrule
\endfirsthead
\toprule
\textbf{Circuit} & $\mathbf{|C|}$ & \multicolumn{2}{c}{$\boldsymbol{\xB}$-values (fraction \& decimal)} \\
\midrule
\endhead
\bottomrule
\endfoot
$\{3, 5, 6, 8, 9\}$ & 5 & $1/6$ & 0.1667 \\
$\{1, 2, 4, 7\}$ & 4 & $13/90$ & 0.1444 \\
$\{1, 2, 4, 9\}$ & 4 & $13/90$ & 0.1444 \\
$\{1, 2, 7, 9\}$ & 4 & $13/90$ & 0.1444 \\
$\{1, 4, 7, 9\}$ & 4 & $13/90$ & 0.1444 \\
$\{2, 4, 7, 9\}$ & 4 & $13/90$ & 0.1444 \\
$\{1, 2, 3, 4, 5, 6, 8\}$ & 7 & $1/36$ & 0.0278 \\
$\{1, 2, 3, 5, 6, 7, 8\}$ & 7 & $1/36$ & 0.0278 \\
$\{1, 3, 4, 5, 6, 7, 8\}$ & 7 & $1/36$ & 0.0278 \\
$\{2, 3, 4, 5, 6, 7, 8\}$ & 7 & $1/36$ & 0.0278 \\
\end{longtable}

\begin{small}
 \bibliographystyle{plain}
 \bibliography{bingobib}

\begin{thebibliography}{10}

\bibitem{Borovik-Gelfand-White}
Borovik AV, Gelfand IM, and White N.
\newblock {\em Coxeter matroids}, volume 216 of {\em Progress in Mathematics}.
\newblock Birkh\"{a}user Boston, Inc., Boston, MA, 2003.

\bibitem{Mayhew2006}
Mayhew D.
\newblock Equitable matroids.
\newblock {\em Electron. J. Combin.}, 13(1):Research Paper 41, 8, 2006.

\bibitem{Mayhew-et-al-2011}
Mayhew D, Newman M, Welsh D, and Whittle G.
\newblock On the asymptotic proportion of connected matroids.
\newblock {\em European J. Combin.}, 32(6):882--890, 2011.

\bibitem{welsh2010matroid}
Welsh DJA.
\newblock {\em Matroid theory}, volume No. 8 of {\em L. M. S. Monographs}.
\newblock Academic Press [Harcourt Brace Jovanovich, Publishers], London-New
  York, 1976.

\bibitem{akrami2025matroidsequitable}
Akrami H, Raj R, and Végh LA.
\newblock Matroids are equitable.
\newblock Preprint available at \url{https://arxiv.org/abs/2507.12100}, 2025.

\bibitem{anderson2002combinatorics}
Anderson I.
\newblock {\em Combinatorics of finite sets}.
\newblock Dover Publications, Inc., Mineola, NY, 2002.
\newblock Corrected reprint of the 1989 edition.

\bibitem{oxley2011}
Oxley JG.
\newblock {\em Matroid theory}.
\newblock Oxford University Press, 2nd edition, 2011.

\bibitem{mayhew2008matroids}
Dillon Mayhew and Gordon~F Royle.
\newblock Matroids with nine elements.
\newblock {\em Journal of Combinatorial Theory, Series B}, 98(2):415--431,
  2008.

\bibitem{ultralogc}
Anari N, Liu K, Gharan SO, and Vinzant C.
\newblock Log-concave polynomials {III}: {M}ason's ultra-log-concavity
  conjecture for independent sets of matroids.
\newblock {\em Proc. Amer. Math. Soc.}, 152(5):1969--1981, 2024.

\bibitem{Branden-Huh20}
Br\"{a}nd\'{e}n P and Huh J.
\newblock Lorentzian polynomials.
\newblock {\em Ann. of Math. (2)}, 192(3):821--891, 2020.

\bibitem{CameronNotes}
Cameron P.
\newblock Notes on matroids and codes.
\newblock Available at
  \url{https://webspace.maths.qmul.ac.uk/p.j.cameron/comb/matroid.pdf}.

\bibitem{nelson2011exponentially}
Nelson P.
\newblock {\em Exponentially Dense Matroids}.
\newblock Ph.{D}.\ thesis, University of Waterloo, 2011.

\bibitem{ConcreteMathematics}
Graham RL, Knuth DE, and Patashnik O.
\newblock {\em Concrete mathematics}.
\newblock Addison-Wesley Publishing Company, Reading, MA, second edition, 1994.
\newblock A foundation for computer science.

\bibitem{unimodal}
Stanley RP.
\newblock Log-concave and unimodal sequences in algebra, combinatorics, and
  geometry.
\newblock {\em Annals of the New York Academy of Sciences}, 576(1):500--535,
  1989.

\bibitem{ligget}
Liggett TM.
\newblock Ultra logconcave sequences and negative dependence.
\newblock {\em Journal of Combinatorial Theory, Series A}, 79(2):315--325,
  1997.

\end{thebibliography}
 \end{small}

\end{document}